\documentclass[american,a4paper]{amsart}
\pdfoutput=1
\usepackage[utf8]{inputenc}
\usepackage{refstyle}
\usepackage{mathrsfs}
\usepackage[inline]{enumitem}
\usepackage{mathtools}
\usepackage{amssymb}
\usepackage{graphicx}
\usepackage[unicode=true,pdfusetitle,
 bookmarks=true,bookmarksnumbered=false,bookmarksopen=false,
 breaklinks=false,pdfborder={0 0 1},backref=false,colorlinks=false]{hyperref}
\usepackage{breakurl}
\usepackage{subcaption}
\usepackage{pgfplots}
\usepackage[foot]{amsaddr}
\usepackage{fancyvrb}

\usepackage[left=100pt, right=100pt]{geometry}

\providecommand{\propositionname}{Proposition}
\providecommand{\theoremname}{Theorem}
\providecommand{\lemmaname}{Lemma}
\providecommand{\remarkname}{Remark}

\numberwithin{equation}{section}
\numberwithin{figure}{section}
\theoremstyle{plain}
\newtheorem{thm}{\protect\theoremname}[section]
\theoremstyle{plain}
\newtheorem*{thm*}{\protect\theoremname}
\theoremstyle{plain}
\newtheorem{prop}[thm]{\protect\propositionname}
\theoremstyle{plain}
\newtheorem{lem}[thm]{\protect\lemmaname}
\theoremstyle{remark}
\newtheorem*{rem*}{\protect\remarkname}

\newref{prop}{name=Proposition~}
\newref{thm}{name=Theorem~}
\newref{lem}{name=Lemma~}
\newref{cor}{name=Corollary~}
\newref{conj}{name=Conjecture~}
\newref{fig}{name=Figure~}
\newref{subsec}{name=Subsection~}
\newref{sec}{name=Section~}
\newref{app}{name=Appendix~}

\begin{document}

\title[A brake to reverse a gene drive invasion]{Catch me if you can: a spatial model for a brake-driven gene drive reversal}

\author{L\'{e}o Girardin}
\address[L. G.]{Laboratoire de Math\'{e}matiques d'Orsay, Universit\'{e} Paris
Sud, CNRS, Universit\'{e} Paris-Saclay, 91405 Orsay Cedex, France}
\email{leo.girardin@math.u-psud.fr}

\author{Vincent Calvez}
\address[V. C.]{Institut Camille Jordan, UMR 5208 CNRS \& Universit\'{e} Claude Bernard Lyon 1,  France}
\email{vincent.calvez@math.cnrs.fr}

\author{Florence D\'{e}barre}
\address[F. D.]{CNRS, Sorbonne Universit\'{e}, Universit\'{e} Paris Est Cr\'{e}teil, Universit\'{e} Paris Diderot, INRA, IRD, Institute of Ecology and Environmental sciences - Paris, IEES-Paris, 75005 Paris, France}
\email{florence.debarre@normalesup.org}

\begin{abstract}
\textcolor{black}{Population management using artificial gene drives (alleles biasing inheritance, increasing their own transmission to offspring) is becoming a realistic possibility with the development of CRISPR-Cas genetic engineering. A gene drive may however have to be stopped. ``Antidotes'' (brakes) have been suggested, but have been so far only studied in well-mixed populations. Here, we consider a reaction--diffusion system modeling the release of a gene drive (of fitness $1-a$) and a brake (fitness $1-b$, $b\leq a$) in a wild-type population (fitness $1$). We prove that, whenever the drive fitness is at most $1/2$ while the brake fitness is close to $1$, coextinction of the brake and the drive occurs in the long run. On the contrary, if the drive fitness is greater than $1/2$, then coextinction is impossible: the drive and the brake keep spreading spatially, leaving in the invasion wake a complicated spatio-temporally heterogeneous genetic pattern. Based on numerical experiments, we argue in favor of a global coextinction conjecture provided the drive fitness is at most $1/2$, irrespective of the brake fitness. The proof relies upon the study of a related predator--prey system with strong Allee effect on the prey. Our results indicate that some drives may be unstoppable, and that, if gene drives are ever deployed in nature, threshold drives, that only spread if introduced in high enough frequencies, should be preferred. }
\end{abstract}

\keywords{long-time behavior, gene drive, brake, predator-prey, strong Allee effect}
\subjclass[2000]{35K57, 37N25, 92D10, 92D25.}

\maketitle

\section{Introduction}
With the development of CRISPR-Cas9 genetic engineering, population management using gene drives has become a realistic possibility. The technique consists in artificially biasing the inheritance of a trait of interest in a target population \cite{Esvelt_2014}. 
Such biased inheritance is due to the presence of an artificial self-replicating element expressing a DNA-cutting enzyme (the Cas9 endonuclease), such that initially heterozygous individuals (\textit{i.e.} carrying the drive construct on one chromosome and a wild-type sequence on the homologous chromosome) produce almost exclusively drive-carrying gametes instead of $50\%$ of drive-carrying gametes, as expected under Mendelian segregation (see \figref{conversion}). The proportion of drive-carrying gametes depends on the conversion efficiency, perfect conversion meaning that an initially heterozygous individual produces $100\%$ of drive-carrying gametes. Thanks to its supra-Mendelian transmission to offspring, a drive can spread in a population even if it confers a significant fitness cost \cite{Deredec_2008,Tanaka_Stone_N,Unckless_2015}. Potential applications for human health and agriculture include the modification of mosquito populations to make them resistant to malaria or the eradication of agricultural pest species \cite{NASEM_2016}. 

\begin{figure}[h!]
    \centering
    \includegraphics[width = 0.5\linewidth]{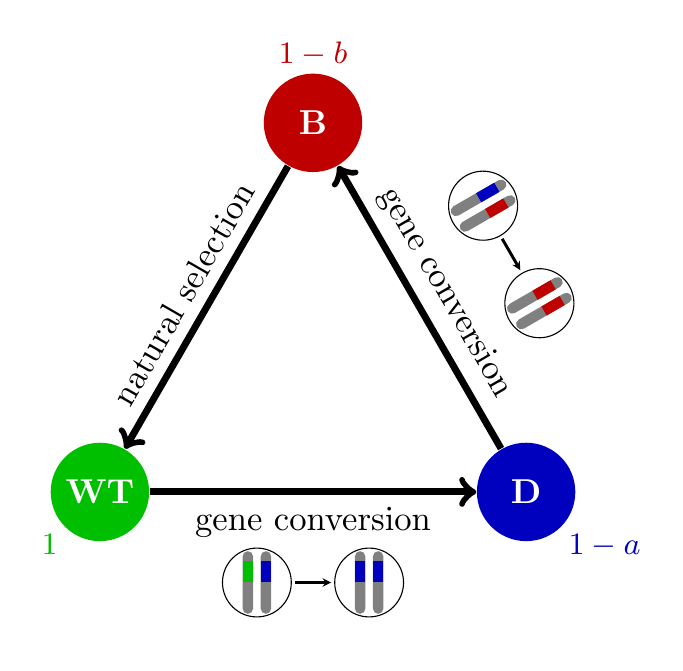}
    \caption{\textcolor{black}{The presence of a brake induces rock-paper-scissors dynamics. The drive wins over the wild-type allele due to gene conversion (such that individuals born as heterozygotes only produce drive-carrying gametes). The brake wins over the drive due to gene conversion and better fitness. The wild-type allele wins over the brake because of greater fitness. }}
    \label{fig:conversion}
\end{figure}

However, while very promising, the technique is not risk-free. A drive could have off-target effects, or spread in a non-target population. For instance, a drive can be introduced on an island to eradicate a local rat population, but the dispersal of drive-carrying individuals to the mainland or to another island would also threaten those populations \cite{Esvelt_2017, Rode_2019}. The effects of population modification using a drive may also have unexpected consequences on other species, \textit{e.g.}, predators or competitors. More generally, it is important to be able to control the spread of a drive, and to stop it if necessary. To this end, a ``brake'' construct was proposed, that does not contain the \textit{cas9} gene (and is hence unable to convert a wild-type allele), but that is able to target the very \textit{cas9} sequence contained in a drive construct, and therefore to convert a drive allele into a brake allele in a (drive/brake) heterozygote \cite{Wu_2016}. The construct has been shown to not only stop a drive, but also in some cases lead to the recovery of the original wild-type population \cite{Vella_2017}. 

Most models of gene drive consider well-mixed populations (except meta-population and partial differential equations \textcolor{black}{(PDE)} models, \cite{Beaghton_2016,Noble_2018,Tanaka_Stone_N}). Here we consider the influence of space and limited dispersal, and ask whether a brake construct is able to stop the spatial spread of a drive. More precisely, we introduce a minimalist PDE model of two interacting sub-populations combining spatial diffusion and Mendelian population genetics. The two sub-populations \textcolor{black}{correspond to the subset of the gene pool with the drive allele, and to the subset with the brake allele (the rest consisting of the wild-type allele).}
The case of a single sub-population of gene drive was addressed in \cite{Tanaka_Stone_N}. It was shown that the spatial invasion of the gene drive allele was successful up to a 
\textcolor{black}{fitness cost $a \simeq 0.7$ associated with the drive allele}
(see below for details). In the present work, we ask whether the brake, even if lately introduced (with a spatial delay with respect to the gene drive invasion), can catch up the invasion and \textcolor{black}{stop the spatial propagation of the gene drive.}
Using phase plane analysis and techniques from the \textcolor{black}{theory of} reaction--diffusion \textcolor{black}{equations}, 
\textcolor{black}{we show that the drive can be stopped} 
under some conditions on the model parameters. 
\textcolor{black}{More precisely, we} prove that
the gene drive frequency will be eventually reduced to zero everywhere, provided that the fitness cost of the drive \textcolor{black}{$a$} is above \textcolor{black}{0.5 (\textit{i.e.}, the fitness associated with the drive allele is $1-a<0.5$)} and that the fitness cost of the brake is small enough.

\subsection{The model}
We adopt the following set of notations:
\begin{itemize}
    \item $N\in\mathbb{N}$ is the spatial dimension (typically $N\in\{1,2,3\}$);
    \item $n\in\left[0,+\infty\right)$ is the total population density;
    \item $u\in\left[ 0,1 \right]$ and $v\in\left[ 0,1 \right]$ ($0\leq u+v\leq 1$) are the respective frequencies in the population
    $n$ of the gene drive allele $D$ and the brake allele $B$ (so that the frequency of the wild-type allele 
    $O$ is exactly $1-u-v$);
    \item $a\in\left( 0,1 \right)$ and $b\in\left( 0,1 \right)$ are the respective selective disadvantage (\textit{i.e.}, decreased survival) of the homozygous individuals $DD$ and $BB$ compared to the wild-type, with the assumption $a\geq b$, which
    is biologically relevant\footnote{Depending on the construct, the brake could just convert a drive without affecting its effect on fitness ($b$ close to $a$, $b \leq a$), or at the other extreme the brake could carry a cargo gene restoring wild-type fitness ($b$ close to $0$).};
    \item $h\in\left[0,1\right]$ is the dominance of the brake allele $B$ on the wild-type allele $O$ (in particular, 
    $B$ is dominant if $h=1$, recessive if $h=0$ and additive if $h=1/2$).
\end{itemize}

We assume that:
\begin{enumerate}[label=(A\theenumi)]
    \item the conversion efficiency of the drive (conversion $OD\to DD$) and 
of the brake (conversion $DB\to BB$) are perfect\footnote{This is merely for algebraic convenience and the general case will be discussed below in \subsecref{discussion_efficiency}.};
    \item gene conversion takes place early in development (\textit{e.g.}, an individual born as $OD$ becomes $DD$ and has the fitness of a $DD$ individual)\footnote{With late gene conversion (typically in the \textcolor{black}{germline}), an $OD$-born individual would have the fitness of an $OD$. In both cases though, only $D$ gametes are produced by this individual.};
    \item \textcolor{black}{individuals mate randomly (random union of gametes,} according to a uniform law).
\end{enumerate}
Then, following the biological literature \cite{Unckless_2015,Vella_2017}, the spatially homogeneous, next-generation discrete system reads:
\begin{equation*}
    \begin{cases}
\displaystyle        1-u_{g+1}-v_{g+1}= \frac{\left(1-u_g-v_g\right)\left(1-u_g-v_g+\left(1-hb\right)v_g\right)}{1-\left( au_g^2+bv_g^2+2bu_g v_g \right)-2\left( 1-u_g-v_g \right)\left( au_g+hbv_g \right)} ,\medskip\\
\displaystyle        u_{g+1} =  \frac{u_g\left(\left( 1-a \right)u_g+2\left( 1-a \right)\left( 1-u_g-v_g \right)\right)}{1-\left( au_g^2+bv_g^2+2bu_g v_g \right)-2\left( 1-u_g-v_g \right)\left( au_g+hbv_g \right)} ,\medskip\\
\displaystyle        v_{g+1} = \frac{v_g\left( \left( 1-b \right)v_g+2\left( 1-b \right)u_g+\left( 1-hb \right)\left( 1-u_g-v_g \right)\right)}{1-\left( au_g^2+bv_g^2+2bu_g v_g \right)-2\left( 1-u_g-v_g \right)\left( au_g+hbv_g \right)} .
    \end{cases}
\end{equation*}
The first line corresponds to the frequency dynamics of the wild-type allele, the second line to the drive, and the last line to the brake. For each equation, the numerator on the right-hand side corresponds to the amount of corresponding alleles produced, while the denominator corresponds to the total amount of alleles (\textit{i.e.}, the ``mean fitness'' in the population). It can be verified that this denominator is exactly such that the first equation is true (in other words, such that the sum of the three frequencies remains identically equal to $1$ as time goes on). Since the first equation is redundant, we get rid of it hereafter. 

These equations can be understood as follows.
\begin{itemize}
    \item Wild-type alleles (first equation) are carried \begin{enumerate*}[label=\textit{(\roman*)}]\item by all of the gametes produced by $OO$ homozygotes (initially in frequency $(1-u_g-v_g)^2$, and with fitness $1$), and \item by half of the gametes of $OB$ heterozygotes (in frequency $2 v_g  (1-u_g-v_g)$, and with fitness $(1-h b)$). Note that since gene conversion is assumed to be perfect, no wild-type alleles are produced by initially $OD$ individuals\end{enumerate*}. 
    \item Drive alleles (second equation) are carried \begin{enumerate*}[label=\textit{(\roman*)}]\item by all of the gametes produced by $DD$ homozygotes (initially in frequency $u_g^2$ and who have a fitness $(1-a)$), but also \item by all of the gametes produced by initially $OD$ heterozygotes (initially in frequency $2 u_g (1-u_g-v_g) $), who were immediately converted into $DD$ homozygotes, and hence have fitness $(1-a)$. Since gene conversion is assumed to be perfect, no drive alleles are produced by initially $DB$ individuals\end{enumerate*}. 
    \item Brake alleles (third line) are carried \begin{enumerate*}[label=\textit{(\roman*)}]\item by all of the $BB$ homozygotes (initially in frequency $v_g^2$, and who have a fitness $(1-b)$), \item by all of the gametes produced by initially $DB$ heterozygotes (in frequency $2 u_g v_g$), who were immediately converted into $BB$ homozygotes, and hence have fitness $(1-b)$, and finally \item by half of the gametes produced by $OB$ heterozygotes (in frequency $2 v_g  (1-u_g-v_g)$, and with fitness $(1-h b)$)\end{enumerate*}. 
    \item \textcolor{black}{Since only ratios of fitness appear in the equations, the assumption
	that the wild-type fitness is unitary is done without loss of generality.}
\end{itemize}

Next, we take the spatial diffusion of the individuals of the population $n$ into account and assume:
\begin{enumerate}[label=(A\theenumi)]
    \setcounter{enumi}{3}
    \item the time scale of the diffusion mechanism and the maturation time between two generations are of the same order;
    \item each sub-population diffuses at the same rate.
\end{enumerate}
We are now in position to perform a classical first-order approximation \cite{Tanaka_Stone_N} and obtain the following
reaction--diffusion PDE system:
\begin{equation*}
    \begin{cases}
\displaystyle        \frac{\partial u}{\partial t} - \Delta u -2\nabla(\log  n)\cdot\nabla u = u\left( \frac{\left( 1-a \right)u+2\left( 1-a \right)\left( 1-u-v \right)}{1-\left( au^2+bv^2+2buv \right)-2\left( 1-u-v \right)\left( au+hbv \right)}-1 \right),\medskip\\
\displaystyle        \frac{\partial v}{\partial t} - \Delta v -2\nabla(\log  n)\cdot\nabla v = v\left( \frac{\left( 1-b \right)v+2\left( 1-b \right)u+\left( 1-hb \right)\left( 1-u-v \right)}{1-\left( au^2+bv^2+2buv \right)-2\left( 1-u-v \right)\left( au+hbv \right)}-1 \right),
    \end{cases}
\end{equation*}
\textcolor{black}{where in general $u$, $v$ and $n$ are functions of time $t$ and space $x$. Note that in this system there is no equation on the
total population density $n$, which is therefore entirely unknown at this point.}

The transport term $2\nabla(\log  n)$ is the signature of gene flow in populations which are not homogeneous in size. This is a consequence of the reformulation of the problem by means of frequencies rather than population densities. To make this connection clear, we point out that the diffusion operator $\frac{\partial}{\partial t}-\Delta$ for the population density $un$ is related to the diffusion--transport operator $\frac{\partial}{\partial t}-\Delta - 2\nabla(\log  n)$ for the frequency $u$:
\begin{align*}
\frac{\partial \left(un\right)}{\partial t}-\Delta \left(un\right) & =\frac{\partial u}{\partial t}n+\frac{\partial n}{\partial t}u-\left(\Delta u\right)n-\left(\Delta n\right)u-2\nabla u\cdot \nabla n\\
& = n\left[\frac{\partial u}{\partial t}-\Delta u -2\nabla(\log  n)\cdot\nabla u\right]\textcolor{black}{+}u\left[\frac{\partial n}{\partial t} -\Delta n\right].
\end{align*}

We shall address two types of questions in our study.  \textcolor{black}{\begin{enumerate*}[label=\textit{(\roman*)}]
    \item How do the allelic frequencies evolve when all three genotypes asymptotically coexist? 
    \item Under which assumptions is the wild-type restored spatially uniformly? 
\end{enumerate*}}

\textcolor{black}{For (i), we will mainly be interested in the propagation speeds of $u$ and $v$. Since these speeds are
strongly impacted by the \textit{a priori} unknown transport term $2\nabla\left(\log n\right)$, we will
simplify this part of the discussion by assuming that $n$ is spatially homogeneous, so that the transport term
vanishes. We are aware this is a very serious restriction but cannot properly conclude without it. 
On the contrary, for (ii), where we are interested in spatially uniform convergence, the transport
term does not matter that much and we can easily handle it, under a mild assumption: 
\begin{enumerate}[label=(A\theenumi)]
    \setcounter{enumi}{5}
    \item $\nabla\left(\log n\right)$ is well-defined and uniformly H\"{o}lder-continuous with an exponent
larger than $\frac12$ (for instance, uniformly Lipschitz-continuous). 
\end{enumerate}
This mathematical assumption should not be understood as being too restrictive: it is satisfied if, for 
instance, $n$ solves a generic
population dynamics reaction-diffusion equation with initial condition bounded above and below by positive constants.
Of course it is satisfied if $n$ is spatially homogeneous, so that we will in any case assume that the above
assumption (A6) is satisfied.} 

For ease of reading, we define the scalar parabolic operator
\[
\mathscr{P}=\frac{\partial}{\partial t} - \Delta - 2\nabla(\log  n)\cdot\nabla
\]
as well as the reaction functions
\[
    w\left( u,v \right)=1-\left( au^2+bv^2+2buv \right)-2\left( 1-u-v \right)\left( au+hbv \right),
\]
\[
    g\left( u,v \right)=
    \textcolor{black}{\begin{pmatrix}
    g_1\left(u,v\right)\\
    g_2\left(u,v\right)
    \end{pmatrix}=}
    \begin{pmatrix}
	\left( 1-a \right)u+2\left( 1-a \right)\left( 1-u-v \right)\\
	\left( 1-b \right)v+2\left( 1-b \right)u+\left( 1-hb \right)\left( 1-u-v \right)
    \end{pmatrix},
\]
\[
    f\left( u,v \right)=
    \textcolor{black}{\begin{pmatrix}
    f_1\left(u,v\right)\\
    f_2\left(u,v\right)
    \end{pmatrix}=}
    \frac{1}{w\left( u,v \right)}g\left( u,v \right)-\begin{pmatrix}1\\1\end{pmatrix},
\]
so that the system finally reads:
\begin{equation}
    \mathscr{P} \begin{pmatrix}u\\v\end{pmatrix}= \begin{pmatrix}u\\v\end{pmatrix}\circ f\left( u,v \right).
    \label{sys:general}
\end{equation}
Here, the fact that the $2\times 2$ linear parabolic operator on the left-hand side is the same on 
both lines is of the utmost importance. 
Mathematically, it is a necessary and sufficient condition to apply several theorems of the standard
parabolic theory, in particular a generalized maximum principle due to Weinberger \cite{Weinberger_1975}
that we will indeed use extensively (it is recalled in \appref{WMP}). Biologically, it means that \textcolor{black}{all} individuals 
move in space similarly: the gene under consideration does not affect the motility of the individuals carrying it. 

Since $u$ and $v$ are frequencies, they \textcolor{black}{should} satisfy $u+v\leq1$ (with $1-\left( u+v \right)$ the 
frequency of the wild-type allele $O$). Therefore it is natural to define the triangle
\[
    \mathsf{T}=\left\{(u,v)\in\left[ 0,1 \right]^2\ |\ u+v\leq1\right\}.
\]
\textcolor{black}{We will indeed verify later on that it is an invariant subset of the phase plane.}

\subsection{Results} 
\textcolor{black}{These results and the discussions of \subsecref{discussion_fitness} are summarized in \figref{behaviors_diagram}. 
Numerical simulations are run in \textit{GNU Octave} \cite{Octave} (semi-implicit finite difference scheme with Neumann boundary conditions, see \appref{numeric}).}

\begin{figure}
    \centering
    \resizebox{.9\linewidth}{!}{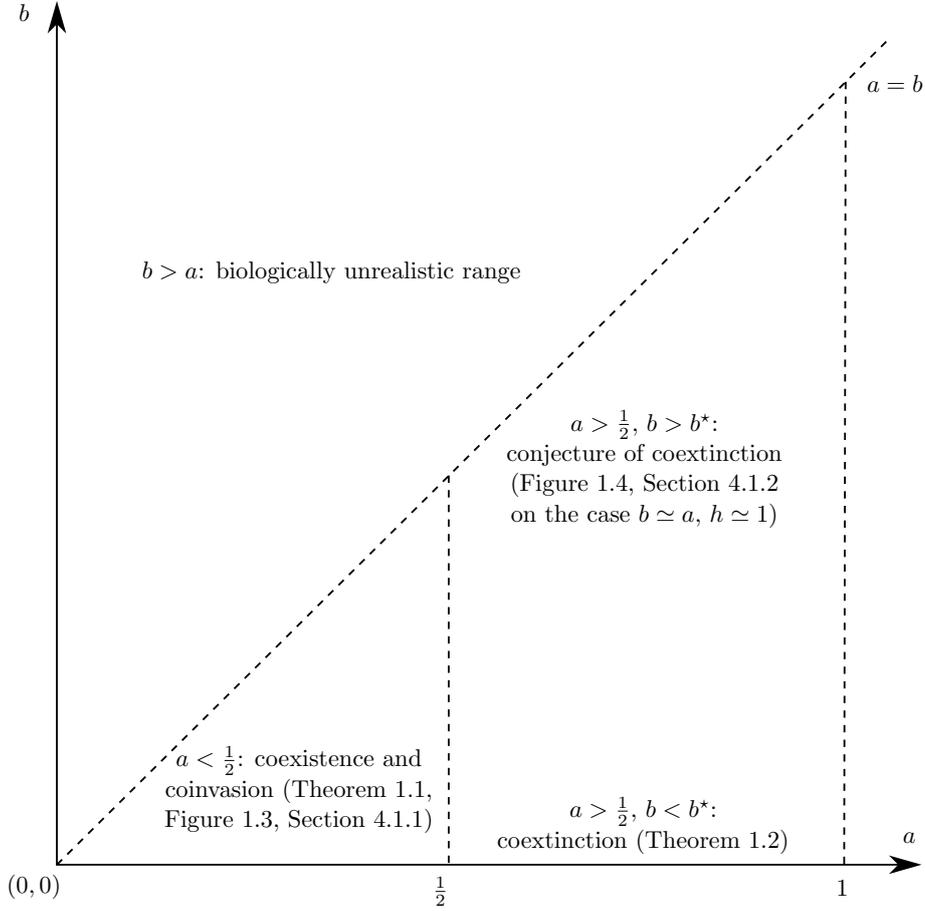}
    \caption{\textcolor{black}{Possible behaviors investigated in the forthcoming pages, depending on the values of $(a,b)$. 
    The parameter $h$ has a fixed arbitrary value in this diagram. The threshold $b^\star$ depends on $a$ and $h$ and its graph as a function of $a$
    is unknown.}}
    \label{fig:behaviors_diagram}
\end{figure}

\subsubsection{Coexistence}

Our first  result deals with the case of coexistence. For technical reasons that will be discussed below, the statement 
is restricted to spatially homogeneous solutions, \textit{i.e.} we consider the solutions of the system of ODE with 
the righ-hand-side of (\ref{sys:general}) only.

In the case where the selective disadvantage of the drive compared to the wild-type is less than a half ($a<1/2$), we 
establish that both sub-populations of drive and brake persist in the long term, as measured by their frequencies 
which are positive at arbitrary large times. The mathematical statement is as follows. 

\begin{thm}\label{thm:persistence_if_a_is_small}
    Let $(u,v)$ be a spatially homogeneous solution of (\ref{sys:general}) \textcolor{black}{with $n$ spatially homogeneous} and 
    initial condition in the interior of $\mathsf{T}$.
    
    If $a<\frac12$, then coextinction of $(u,v)$ cannot occur. More precisely,
    \[
	\limsup_{t\to+\infty}\left(u\left( t \right)+v\left( t \right)\right)>0.
    \]
\end{thm}
In view of numerical experiments, we expect damped oscillations converging to a coexistence
state or sustained oscillations approaching a periodic limit cycle or a heteroclinic cycle 
(see \subsecref{linear_stability} below). 

Concerning the spatially heterogeneous problem, 
describing the invasion of the drive and the brake in a territory occupied by the wild-type allele, 
we explain in \subsecref{discussion_fitness} why we expect that the brake always catches up with the 
drive and that, afterwards, both populations persist in the wake of the joint invasion front. 
We give evidence that this claim is  true if $a\leq1/4$ and the total population $n$ is spatially 
homogeneous, based on Ducrot--Giletti--Matano \cite{Ducrot_Giletti_Matano_2}. We believe it remains 
true even if $1/4<a<1/2$. \figref{behavioursA} is a numerical illustration of this claim. 
Indeed, we observe that the brake catches up with the drive, even if it starts with a space and time delay.
However, the drive is ``strong'' enough to persist, resulting in a joint invasion front followed 
by a complicated spatio-temporal pattern.

\subsubsection{Coextinction}

Our second result deals with the case of coextinction. Here we are able to handle spatially 
heterogeneous solutions and the transport term $2\nabla\left(\log n\right)$.  

We assume that the initial frequencies $\left(u_0,v_0\right)$ are distributed such 
that some individuals carrying the brake allele have been released somewhere, whereas 
the gene drive allele has not completed its invasion in the whole space yet (only individuals 
carrying the wild-type allele  are present far away in space). 
\textcolor{black}{Note that we do not need to assume that the brake is released in a particular region of space -- say, the region of space already colonized by the gene drive.}
If the \textcolor{black}{selective disadvantage of the drive is sufficient}
($a>1/2$) \textcolor{black}{while that } of the brake is not too large ($b<b^{\star}$), then the 
drive goes extinct everywhere, followed by the complete extinction of the brake as well. The 
threshold for the brake $b^{\star}(a,h)$ appears for technical reasons at several steps of our argument. 
Consequently, an explicit value is not straightforward to obtain, and might not be informative. 
\textcolor{black}{As explained below,} we believe that this is only a technical restriction, and that the result should remain 
true \textcolor{black}{in the whole range $b\in\left(0,a\right)$}. The mathematical statement is as follows.

\begin{thm}\label{thm:convergence_to_0_if_b_is_small}
    Let $(u,v)$ be the solution of (\ref{sys:general}) with initial condition $\left(u_0,v_0\right)\in\mathscr{C}\left(\mathbb{R}^N,\mathsf{T}\right)$ satisfying
    \[
	v_0\neq 0\quad\text{and}\quad\lim_{\|x\|\to+\infty}\left( u_0,v_0 \right)(x)=\left( 0,0 \right).
    \]
    
    There exists $b^\star\in(0,a)$ \textcolor{black}{depending on $a$ and $h$} such that, if $a>\frac12$ and $b<b^\star$, then coextinction occurs:
    \[
	\lim_{t\to+\infty}\left(\sup_{x\in\mathbb{R}^N}u\left( t,x \right)+\sup_{x\in\mathbb{R}^N}v\left( t,x \right)\right)=0.
    \]
\end{thm}

It is important to note that the convergence to zero of both $u$ and $v$ is uniform in space, that is, 
we rule out the possibility of a persistent wave of $u$ followed or replaced by a wave in $v$. 

We performed numerical experiments to explore the possible behaviours \textcolor{black}{when $b$ varies in the range $\left(0,a\right)$}. 
We observed complete extinction for values of $b$ up to $a$ (\figref{behavioursB})
\footnote{\textcolor{black}{We actually observed that the extinction threshold is larger than $a$ 
but smaller than $1$ (\figref{behaviours}). Nevertheless, the case $b>a$ is 
beyond the scope of our assumptions and does not correspond to an relevant case in the 
context of the biological problem.}}. To complete the numerical investigation, we give some evidence in \subsecref{discussion_fitness} 
supporting the conjecture that the coextinction of $u$ and $v$ occurs for all $b\leq a$.

\begin{figure}
    \centering
        \begin{subfigure}{0.45\linewidth}
    \includegraphics[width=\linewidth]{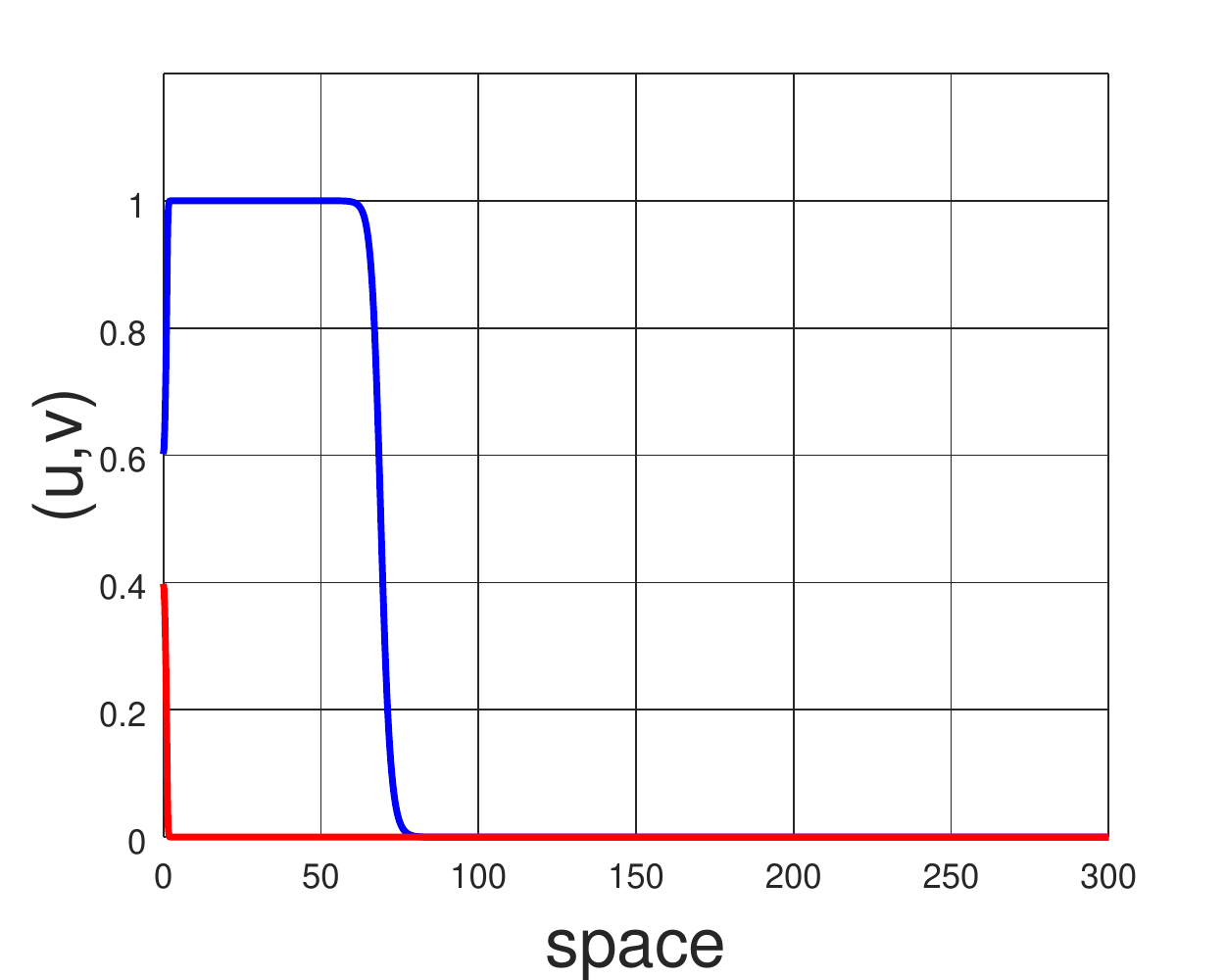}
    \caption{$t = 80$}
    \end{subfigure}
    \begin{subfigure}{0.45\linewidth}
    \includegraphics[width=\linewidth]{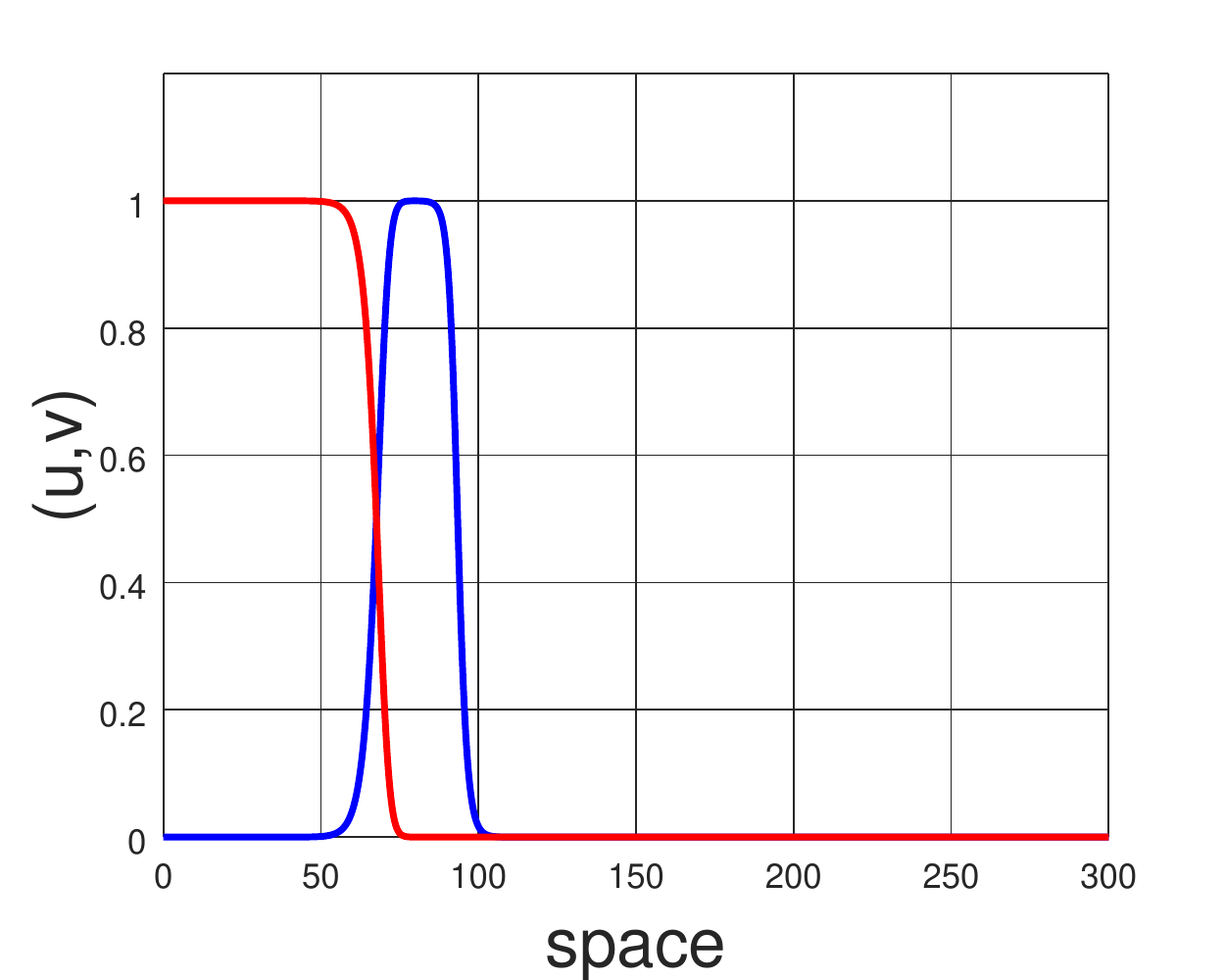}
    \caption{$t = 112$}
    \end{subfigure}
    \begin{subfigure}{0.45\linewidth}
    \includegraphics[width=\linewidth]{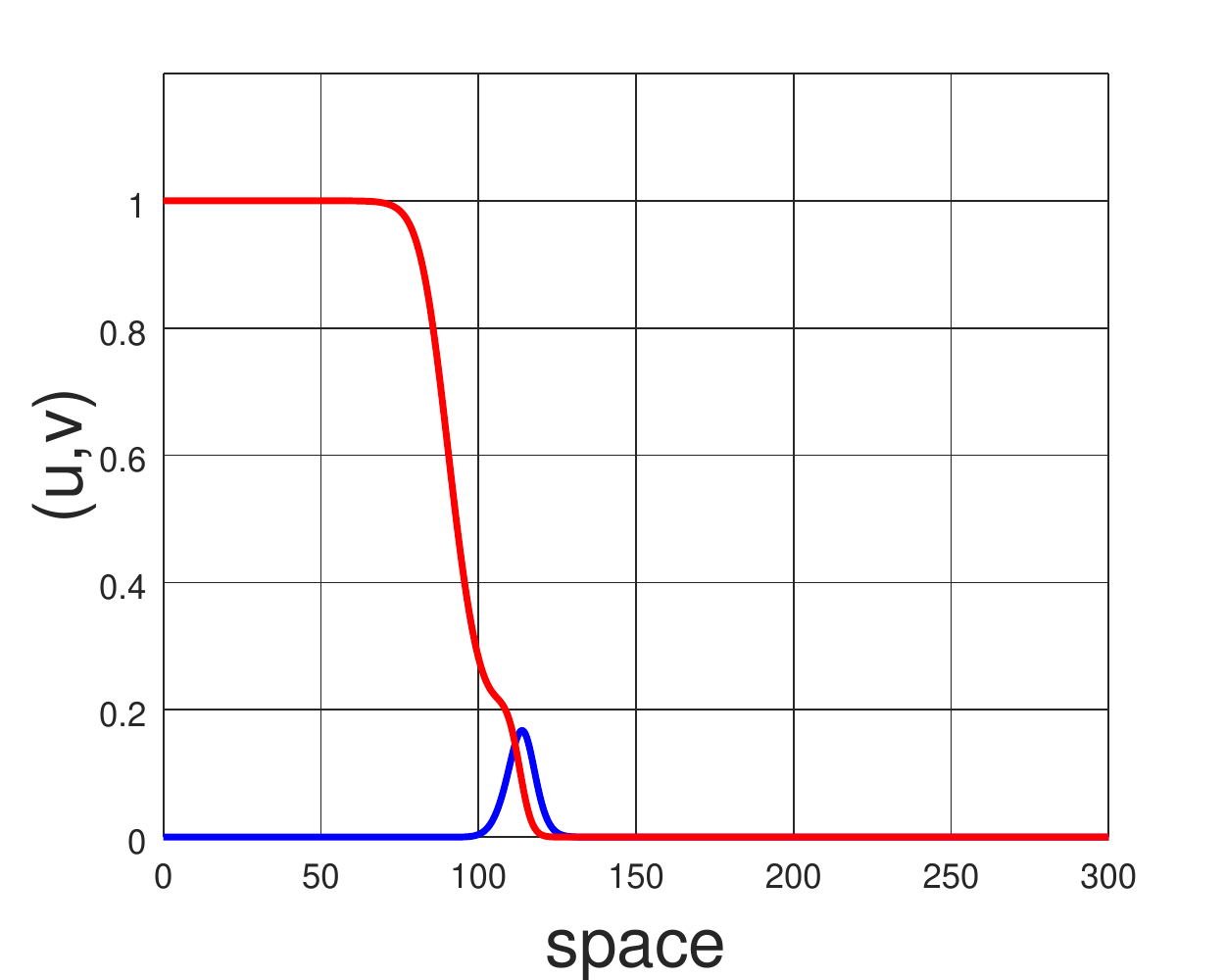}
    \caption{$t = 144$}
    \end{subfigure}
    \begin{subfigure}{0.45\linewidth}
    \includegraphics[width=\linewidth]{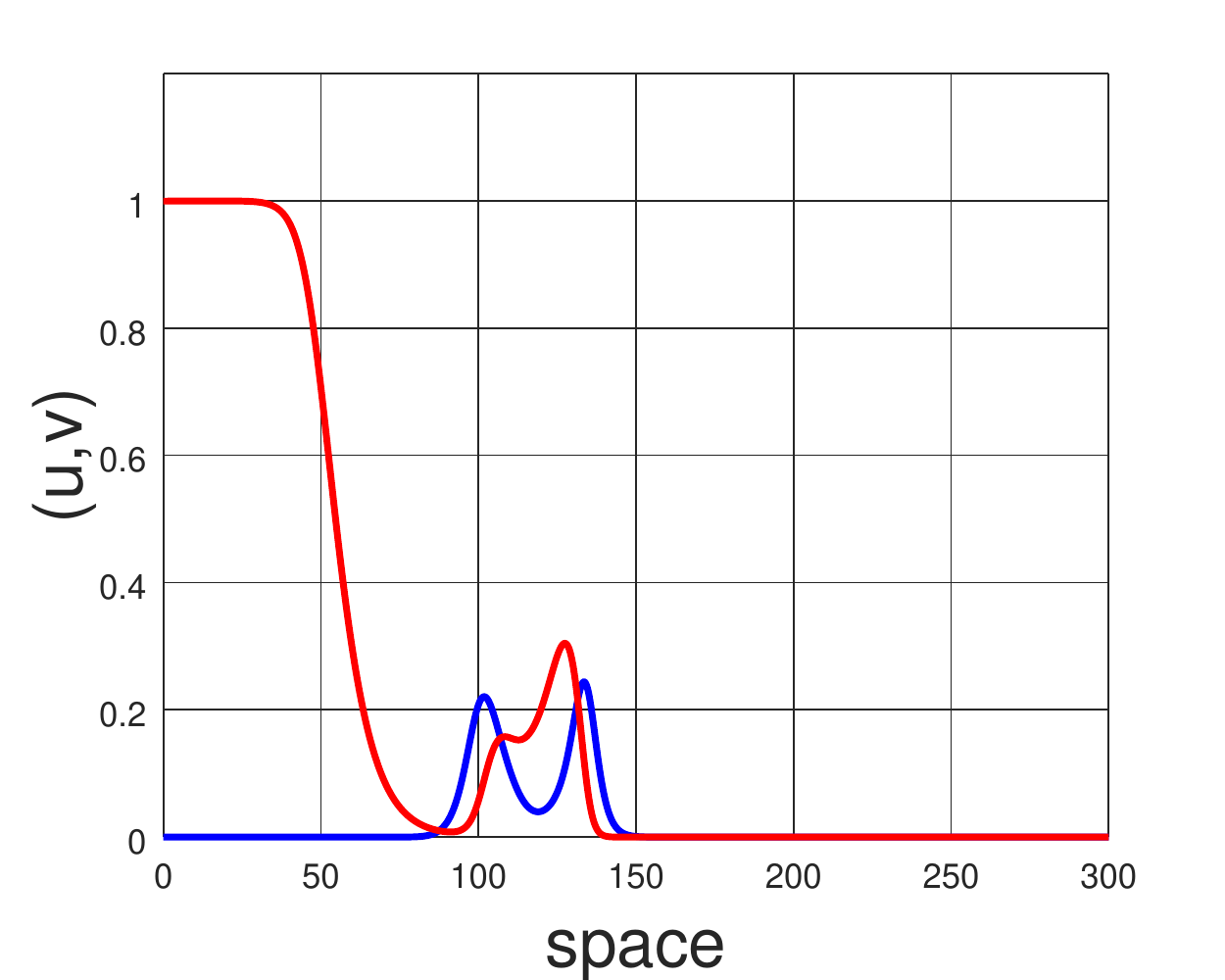}
    \caption{$t = 176$}
    \end{subfigure}
    \begin{subfigure}{0.75\linewidth}
    \includegraphics[width=\linewidth]{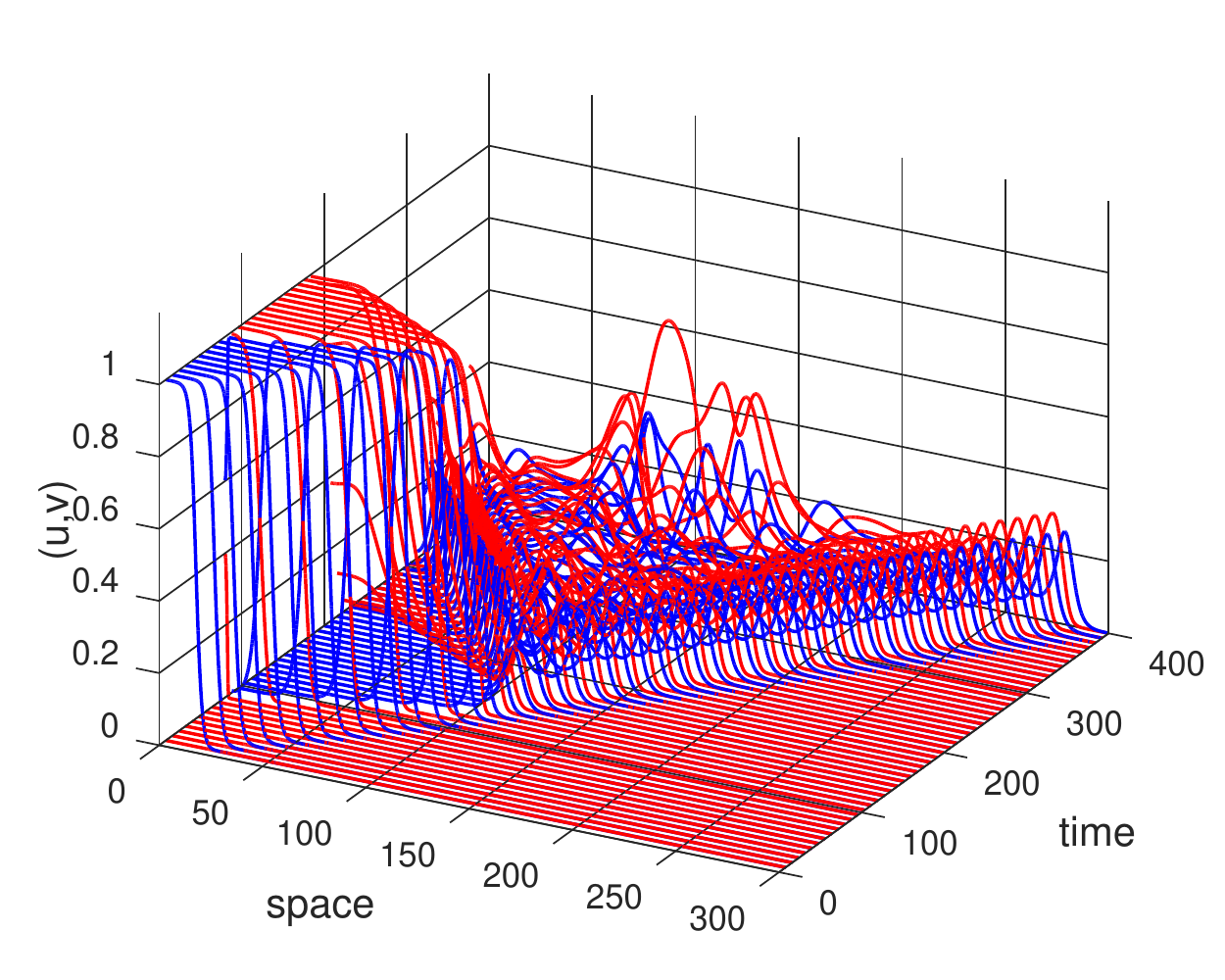}
    \caption{Space-time representation of the solution $(u,v)$}
    \end{subfigure}
    \caption{Numerical solution of (\ref{sys:general}) in the case $a= 0.45<1/2$, $b = 0.35<a$, $h = 0.5$, $n$ spatially homogeneous (\textit{i.e.} the drift term $\nabla\log n$ vanishes). (A-B-C-D) are successive snapshots of the drive $u$ (blue) and the brake $v$ (red) from the time $t=80$ at which the brake $v$ is released on the left-hand-side of the domain. (E) is the superposition of many snapshots, in order to visualize the spatio-temporal dynamics. \textcolor{black}{We observe a small, stable, composite wave followed by spatio-temporal oscillations.}}
    \label{fig:behavioursA}
\end{figure}    

\begin{figure}
    \centering
        \begin{subfigure}{0.45\linewidth}
    \includegraphics[width=\linewidth]{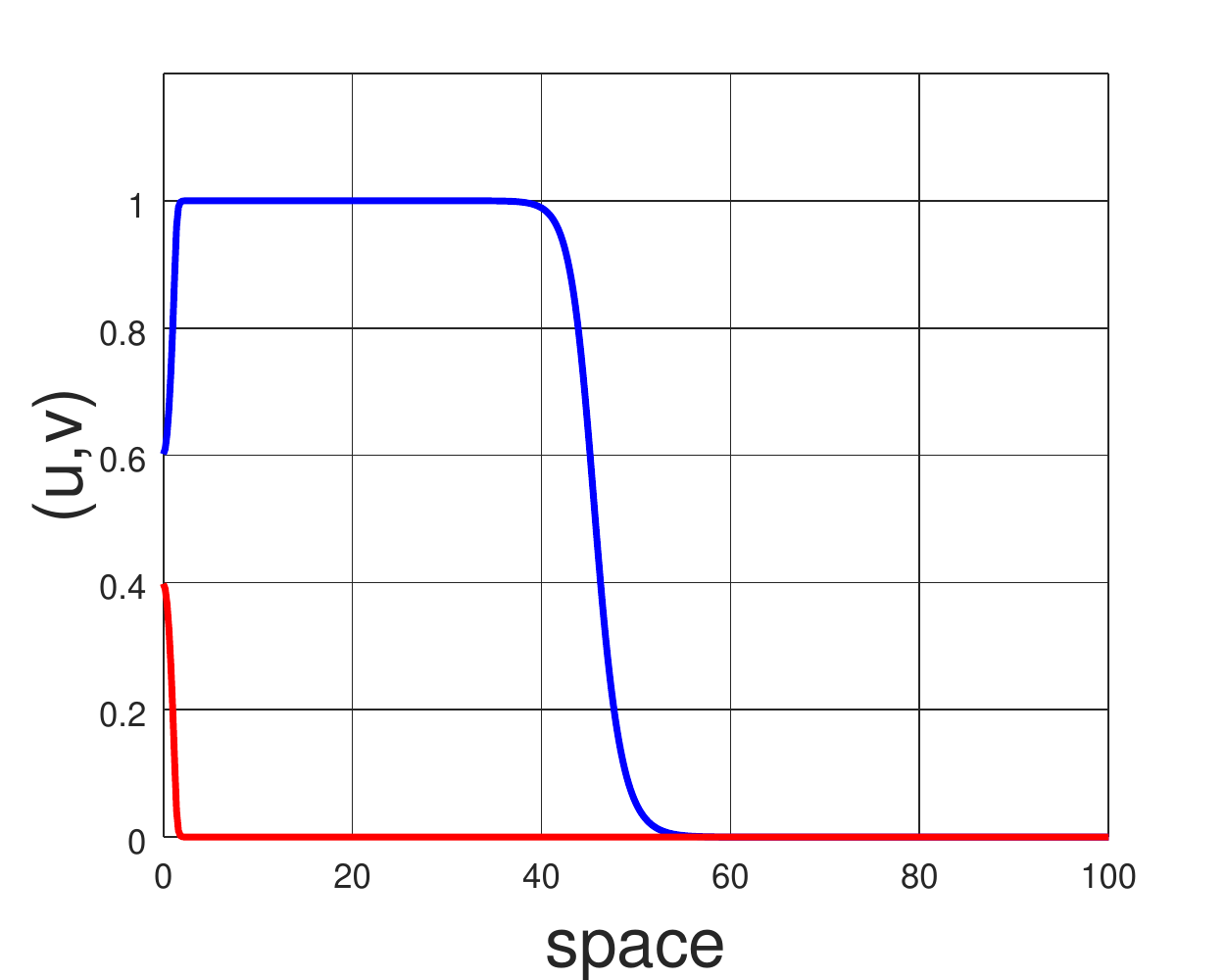}
    \caption{$t = 80$}
    \end{subfigure}
    \begin{subfigure}{0.45\linewidth}
    \includegraphics[width=\linewidth]{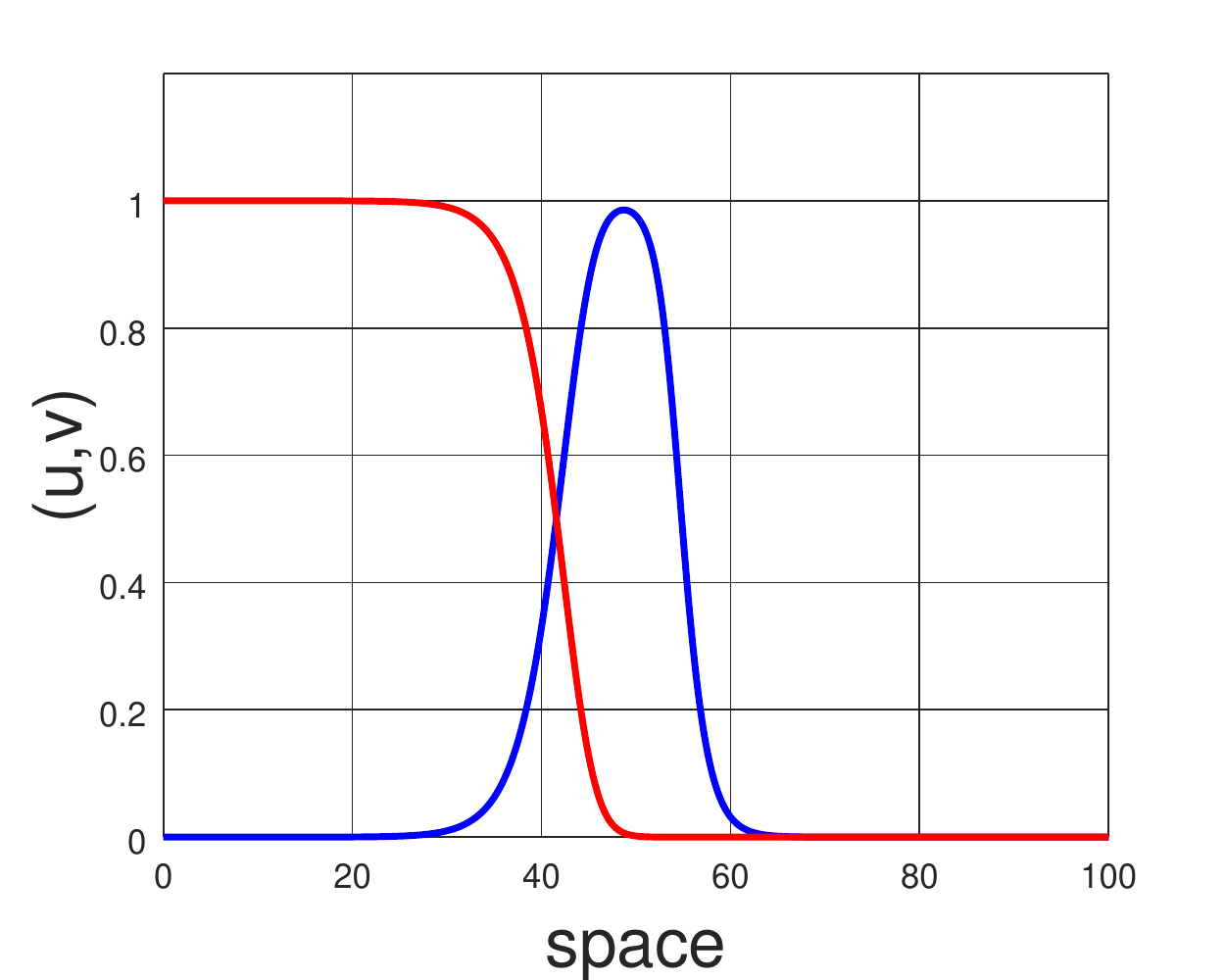}
    \caption{$t = 100$}
    \end{subfigure}
    \begin{subfigure}{0.45\linewidth}
    \includegraphics[width=\linewidth]{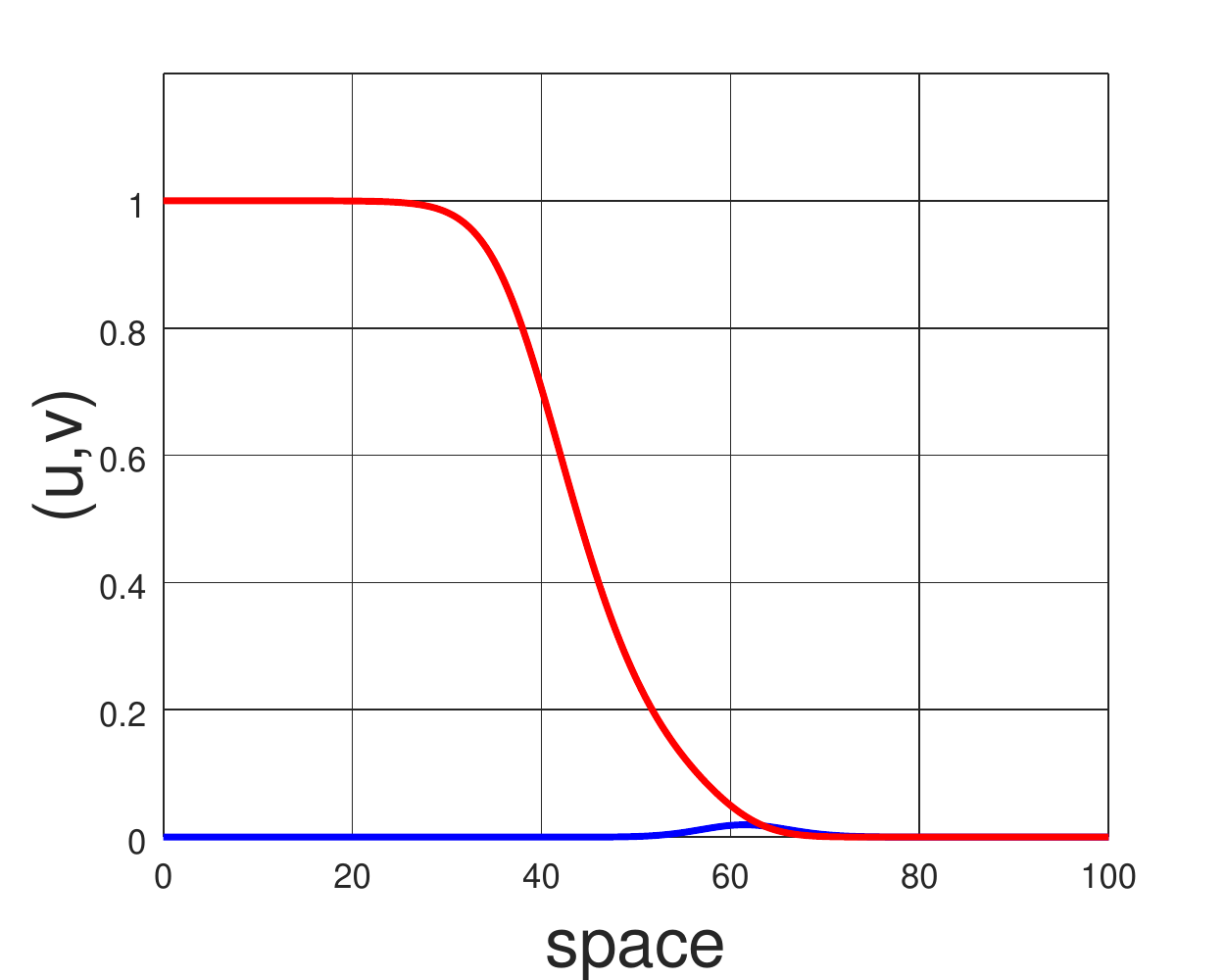}
    \caption{$t = 120$}
    \end{subfigure}
    \begin{subfigure}{0.45\linewidth}
    \includegraphics[width=\linewidth]{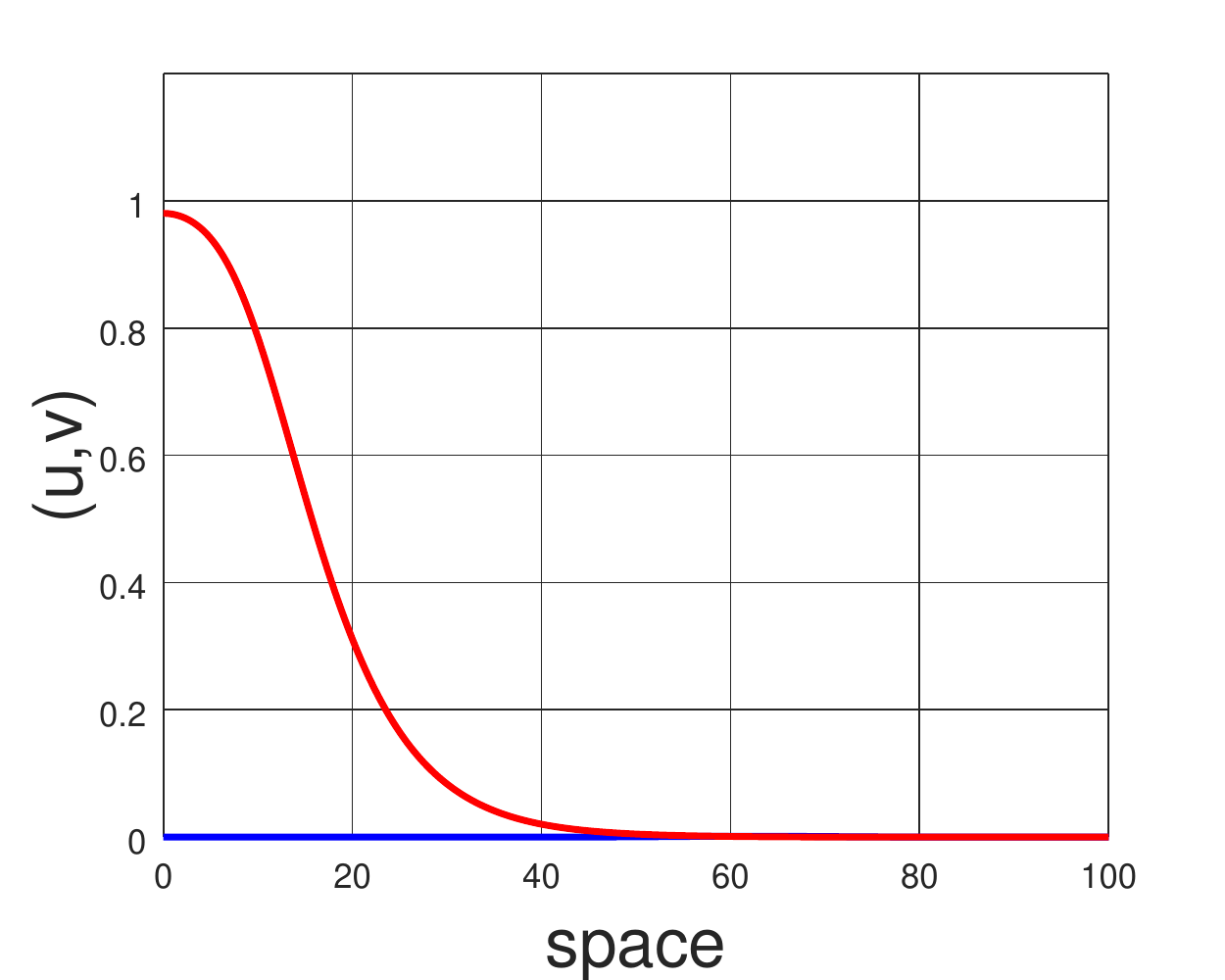}
    \caption{$t = 140$}
    \end{subfigure}
        \begin{subfigure}{0.75\linewidth}
    \includegraphics[width=\linewidth]{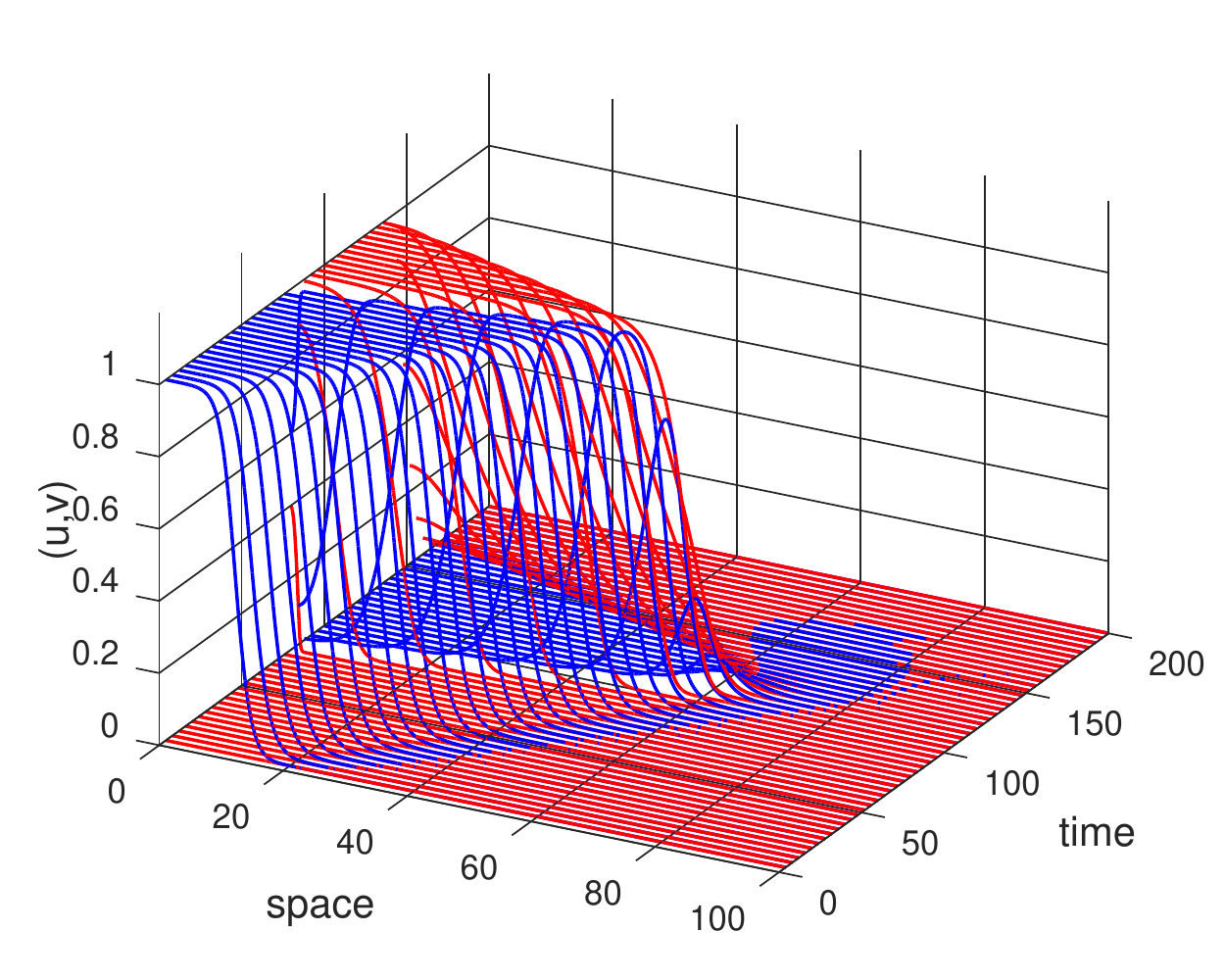}
    \caption{Space-time representation of the solution $(u,v)$}
    \end{subfigure}
    \caption{Same as in Figure \ref{fig:behavioursA}, but with $a = 0.55$ and $b = 0.45$.
    \textcolor{black}{We observe invasion of the drive by the brake, up to extinction of the drive, followed
    by the retraction of the brake, up to complete extinction (meaning restauration of the wild-type).}}\label{fig:behavioursB}
\end{figure}    

\begin{figure}
    \centering
    \begin{subfigure}{0.48\linewidth}
    \includegraphics[width=\linewidth]{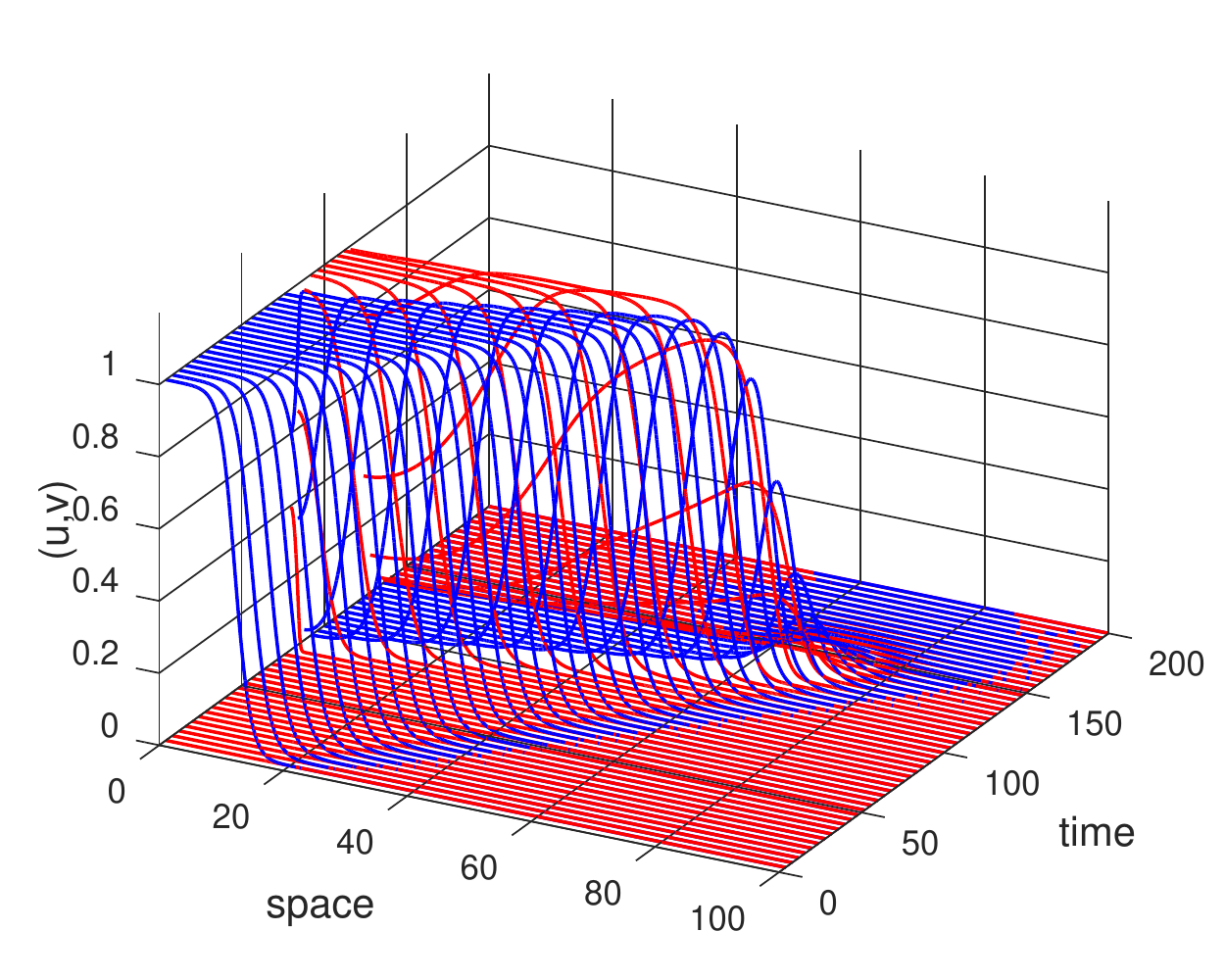}
    \caption{$a= 0.55$, $b = 0.65$}
    \end{subfigure}
    \begin{subfigure}{0.48\linewidth}
    \includegraphics[width=\linewidth]{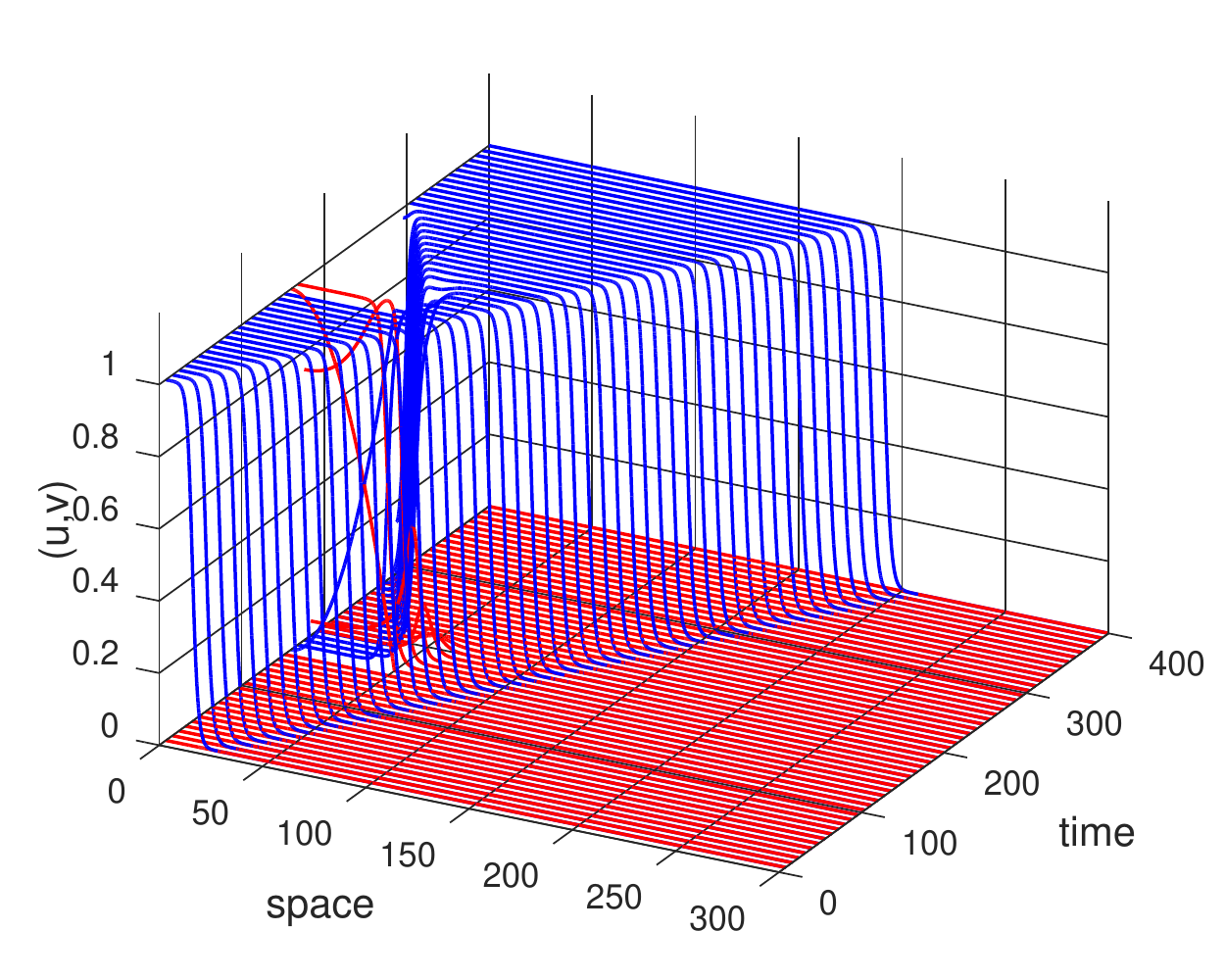}
    \caption{$a= 0.55$, $b = 0.8$}
    \end{subfigure}
    \caption{Further numerical investigation for a \textcolor{black}{larger} disadvantage of the brake $b>a>1/2$ \textcolor{black}{(biologically unrealistic range)}. Uniform coextinction remains true for values of $b$ larger than $a$ (A), but no too large (B). Joint persistence of both species associated with complicated spatio-temporal behaviours were observed for intermediate values of $b$ ($a = 0.55, b = 0.75$, result not shown).}
    \label{fig:behaviours}
\end{figure}

\subsection{Structure of the paper}
Section 2 is devoted to some technical preliminaries and contains in particular an elementary proof
of \thmref{persistence_if_a_is_small}. Section 3 contains the proof of
\thmref{convergence_to_0_if_b_is_small}. \textcolor{black}{In Section 4, }possible extensions are discussed 
and the aforementioned conjecture is explained. \textcolor{black}{Readers not interested in the mathematical proofs can safely skip Sections 2 and 3.}

\section{Preliminaries}
\subsection{Well-posedness}
\textcolor{black}{In this subsection we briefly verify that the model is well-posed mathematically (the denominator
$w(u,v)$ has constant sign) and biologically (the allelic frequencies $u$ and $v$ remain nonnegative and always satisfy $u+v\leq 1$, namely $(u,v)\in\mathsf{T}$).}

\begin{prop}
    The function $w$ satisfies 
    \[
	\max_{\mathsf{T}}w=1\text{ and }\min_{\mathsf{T}}w= 1-a >0.
    \]
    Therefore the system (\ref{sys:general}) is well-posed in $\mathsf{T}$.
\end{prop}
\begin{proof}
    The function $u\mapsto w\left( u,0 \right)$ coincides with 
    $u\mapsto 1-au^2-2au\left( 1-u \right)$, that is $u\mapsto 1-a+a\left( 1-u \right)^2$, 
    whose minimum and maximum in $\left[ 0,1 \right]$ are respectively $1-a>0$ and $1$.
    Since the bound $w\leq1$ in $\mathsf{T}$ is obvious, this directly shows that the maximum of 
    $w$ in $\mathsf{T}$ is indeed $1$. It only remains to confirm that the minimum is indeed $1-a$.
    We prove this claim in two steps: first, we prove that the minimum of $w$ on $\partial\mathsf{T}$ 
    is $1-a$; second, we prove that there is no critical point of $w$ in the interior of $\mathsf{T}$.

    The function $v\mapsto w\left( 0,v \right)$ coincides with 
    $v\mapsto 1-bv^2-2hbv\left( 1-v \right)$, that is $v\mapsto 1+\left( 2h-1 \right)bv^2-2hbv$.
Its derivative with respect to $b$ is $2b\{v(h-1) + h(v-1)\}\leq 0$. Therefore, its minimum in $[0,1]$   is
  attained at $v=1$ and is  $1-b>0$.

    The function $u\mapsto w\left( u,1-u \right)$ coincides with
    $u\mapsto 1-au^2-b\left( 1-u \right)^2-2bu\left( 1-u \right)$, that is 
    $u\mapsto 1-b-\left( a-b \right)u^2$, whose minimum in $\left[ 0,1 \right]$ is $1-a>0$.

    Therefore the minimum of $w$ on $\partial\mathsf{T}$ is indeed $1-a>0$.

    Next, straighforward algebra shows that any critical 
    point $\left( u^\star,v^\star \right)\in\mathbb{R}^2$ satisfies
    \[
	\begin{pmatrix}
	    a & a-\left( 1-h \right)b \\
	    a-\left( 1-h \right)b & \left( 2h-1 \right)b
	\end{pmatrix}
	\begin{pmatrix}
	    u^\star \\ v^\star
	\end{pmatrix}
	=
	\begin{pmatrix}
	    a \\ hb
	\end{pmatrix}.
    \]
    The determinant of the $2\times2$ matrix above is $-a\left( a-b \right)-\left( 1-h \right)^2b^2$,
    which vanishes if and only if $a=b$ and $h=1$, but in such a case the system reduces to 
    $u^\star+v^\star=1$ so that there are no critical points in the interior of
    $\mathsf{T}$. If $a>b$ or $h<1$, the $2\times2$ matrix is invertible
    and Cramer's rule yields
    \begin{align*}
    \left( u^\star,v^\star \right) & =\frac{1}{-a\left( a-b \right)-\left( 1-h \right)^2b^2}
    \begin{pmatrix}
	    a(2h-1)b-hb\left(a-(1-h)b\right) \\ ahb-\left(a-(1-h)b\right)a
	\end{pmatrix}\\
	& =\frac{1}{a\left( a-b \right)+\left( 1-h \right)^2b^2}
    \begin{pmatrix}
	    b(1-h)(a-hb)\\ a(a-b)
	\end{pmatrix}
    \end{align*}
    This point is in the interior of $\mathsf{T}$ if and only if $b(1-h)(a-hb)<(1-h)^2 b^2$,
    that is if and only if $a-hb<(1-h)b$, that is if and only if $a<b$. Hence it is not in
    the interior of $\mathsf{T}$.

    Therefore the minimum of $w$ in $\mathsf{T}$ is attained only on the boundary and is $1-a>0$.
\end{proof}

\begin{prop}
    Any solution of (\ref{sys:general}) with initial condition in
    $\mathscr{C}\left(\mathbb{R}^N,\mathsf{T}\right)$ is valued
    in $\mathsf{T}$ at all times $t\geq 0$.
\end{prop}
\begin{proof}
    To show that $\mathsf{T}$ is an invariant region of the phase space, we use
    Weinberger's maximum principle \cite{Weinberger_1975}. Indeed, the triangle $\mathsf{T}$ is
    a convex invariant set satisfying the so-called slab condition. Therefore we only have to verify
    that the reaction term is inward-pointing on the boundary of $\mathsf{T}$, namely:
    \[
    \begin{pmatrix}u\\0\end{pmatrix}\circ f\left( u,0 \right)\cdot\begin{pmatrix}0\\1\end{pmatrix}\geq 0
    \text{ for all $u\in\left[0,1\right]$},
    \]
    \[
    \begin{pmatrix}0\\v\end{pmatrix}\circ f\left( 0,v \right)\cdot\begin{pmatrix}1\\0\end{pmatrix}\geq 0
    \text{ for all $v\in\left[0,1\right]$},
    \]
    \[
    \begin{pmatrix}u\\1-u\end{pmatrix}\circ f\left( u,1-u \right)\cdot\begin{pmatrix}1\\1\end{pmatrix}\leq 0
    \text{ for all $u\in\left[0,1\right]$}.
    \]
    These three conditions are trivially verified (the left-hand side being always zero).
\end{proof}

\subsection{The propagation of the gene drive alone\label{subsec:without_brake}}
In what follows, we fix $v=0$ and we briefly review some results described in \cite{Tanaka_Stone_N} about the dynamics of the gene drive invasion.

If $v=0$, then (\ref{sys:general}) reduces to 
\begin{align*}
    \mathscr{P} u & = u f_1(u,0)\\
    & = u\frac{-au^2+\left( 3a-1 \right)u-\left( 2a-1 \right)}{1-a+a\left( 1-u \right)^2}\\
    & = \frac{au\left( 1-u \right)\left( u-\frac{2a-1}{a} \right)}{1-a+a\left( 1-u \right)^2}.
\end{align*}

If $a\leq\frac12$, this is a monostable equation. Additionally, a bit of algebra shows that 
this is an equation of KPP type if and only if
$a\leq\frac{1}{4}$, the KPP property being here understood in the following weak sense: the 
maximal growth rate per capita corresponds to sparse populations, namely 
\[
f_1(0,0)= \max_{u\in\left[ 0,1 \right]}f_1(u,0).
\]

If $a>\frac12$, this is a bistable equation with stable steady states $0$ and $1$ and unstable 
intermediate 
steady state $\theta=\frac{2a-1}{a}\in\left( 0,1 \right)$. All known results on bistable equations, and in 
particular \cite{Fife_McLeod_19}, can therefore be applied to this case. 

In particular, the sign of the following quantity plays a crucial role:
\[
    \int_0^1 \frac{au\left( 1-u \right)\left( u-\frac{2a-1}{a} \right)}{1-a+a\left( 1-u \right)^2}\textup{d}u = \frac{\sqrt{1-a}}{a^{3/2}}\operatorname{arctan}\left( \sqrt{\frac{a}{1-a}} \right)-\frac12 - \frac{1-a}{a}
\]
(the calculation of the integral is not detailed here).
Since this is positive if $a=\frac12$, negative if $a=1$ and since
\[
    \frac{\partial}{\partial a}\left( a\mapsto\operatorname{arctan}\left( \sqrt{\frac{a}{1-a}} \right)-\left( 1-\frac{a}{2} \right)\sqrt{\frac{a}{1-a}} \right)=-\frac{\left( 2a-1 \right)a^{3/2}}{4a\left( 1-a \right)^{3/2}}<0,
\]
there exists a unique $a_0\in\left(\frac12,1\right)$ such that the integral is positive
if $a\in\left[\frac12,a_0\right)$, zero if $a=a_0$ and negative if $a\in\left(a_0,1\right]$.
It satisfies the numerical approximation $a_0\simeq 0.6965$.

As a consequence, in the simplified case where $\mathscr{P}=\partial_t -\Delta$ (\textit{i.e.} the total population 
$n$ is spatially homogeneous), 
solutions $u$ that are initially compactly supported will always go extinct if 
$a\in\left(a_0,1\right]$ and will spread and invade if $a\in[0,a_0)$. If $a\in\left(\frac12,a_0\right)$,
it is also necessary that the initial condition is favorable enough (\textit{i.e.} larger than $\frac{2a-1}{a}$ in a wide region -- see the role of the initial data in the emergence of a wave in the bistable case \cite{Barton_Turelli}, and in particular the existence of bubble-like solutions that can prevent the propagation).

\subsection{Basic phase-plane analysis: spatially uniform stationary states\label{subsec:linear_stability}}
\textcolor{black}{This subsection is devoted to the stability analysis of spatially homogeneous solutions. Interestingly,
tedious computations of Jacobians are unnecessary.}

\subsubsection{\textcolor{black}{Boundary stationary states}}
Similarly to \subsecref{without_brake}, if $u=0$, (\ref{sys:general}) reduces to
\[
\mathscr{P}v=-bv\frac{\left(1-v\right)\left( h(1-v) + v(1-h)\right)}{1-bv^2-2hbv(1-v)}.
\]
The right-hand side has exactly the sign of $-v(1-v)$, whence this is a ``backward-monostable'' equation,
where $0$ is stable and $1$ is unstable. 

Again similarly, if $u+v=1$, the equation satisfied by $u$ reduces to
\[
\mathscr{P}u=-u\frac{\left(1-u\right)\left(1-b+\left(a-b\right)u\right)}{1-b-(a-b)u^2}
\]
and this is also a ``backward-monostable'' equation.

\textcolor{black}{These observations together with the stable and unstable manifold theorems} show that, regarding the diffusionless system, 
\begin{itemize}
    \item $(0,0)$ is a stable node if $a>\frac12$ and is a saddle if $a<\frac12$;
    \item $(1,0)$ is a saddle;
    \item $(0,1)$ is a saddle.
\end{itemize}

\textcolor{black}{Now we turn to the stability of $\left(\frac{2a-1}{a},0\right)$ (which is relevant only if $a>\frac12$ so this is assumed below
without loss of generality). 
As $\left(u,v\right)\to\left(0,0\right)$,
\[
\begin{pmatrix}\frac{2a-1}{a}+u\\v\end{pmatrix}\circ f\left( \frac{2a-1}{a}+u,v \right)
\sim\begin{pmatrix}\frac{2a-1}{a}\partial_u f_1\left(\frac{2a-1}{a},0\right) & 
\frac{2a-1}{a}\partial_v f_1\left(\frac{2a-1}{a},0\right)\\ 0 & f_2\left(\frac{2a-1}{a},0\right)\end{pmatrix}
\begin{pmatrix}u\\v\end{pmatrix}
\]
with $\partial_u f_1\left(\frac{2a-1}{a},0\right)>0$ (see \subsecref{without_brake}) and
\[
f_2\left(\frac{2a-1}{a},0\right)=\frac{\left(1-b+a-b\right)\left(2a-1\right)+\left(1-hb\right)\left(1-a\right)}{a\left(1-a-a\left(1-\frac{2a-1}{a}\right)^2\right)}>0.
\]
Hence $(\frac{2a-1}{a},0)$ is an unstable node.}

All this implies that, for the trajectories we have in mind (\textit{i.e.} contained in the interior of $\mathsf{T}$),
convergence to $(1,0)$, $(0,1)$ or $(\frac{2a-1}{a},0)$ is impossible and convergence to $(0,0)$ is possible
if and only if $a\geq\frac12$. Consequently, \thmref{persistence_if_a_is_small} is proved. 

\subsubsection{\textcolor{black}{Interior stationary states}}
Any stationary state in the interior of $\mathsf{T}$ is a solution of the following algebraic system:
\begin{equation*}
    \begin{cases}
        g_1(u,v)=g_2(u,v),\\
        g_1(u,v)=w(u,v),
    \end{cases}
\end{equation*}
which can be solved explicitly ($g_1-g_2$ is a first-order polynomial while $w-g_1$ is a second-order 
polynomial). We do not perform this resolution here, as it is tedious and useless. We simply point out that:
\begin{itemize}
    \item on one hand, the proof of \thmref{convergence_to_0_if_b_is_small} will imply directly the
    nonexistence of such a coexistence state when $a>\frac12$ and $b<b^\star$;
    \item on the other hand, when $a<\frac12$, the flow is rotating anticlockwise on the boundary of
    $\mathsf{T}$, whence by classical 
    phase-plane arguments (\textit{e.g.} the Poincar\'{e}--Bendixson theorem) there exists such a coexistence state.
    Numerically, we observe that this stationary state is unique and is a stable or unstable spiral, depending on the parameters
    (see \figref{phase_portaits_small_a}).
\end{itemize}

\begin{figure}
    \begin{subfigure}{.9\linewidth}
        \resizebox{\linewidth}{!}{\input{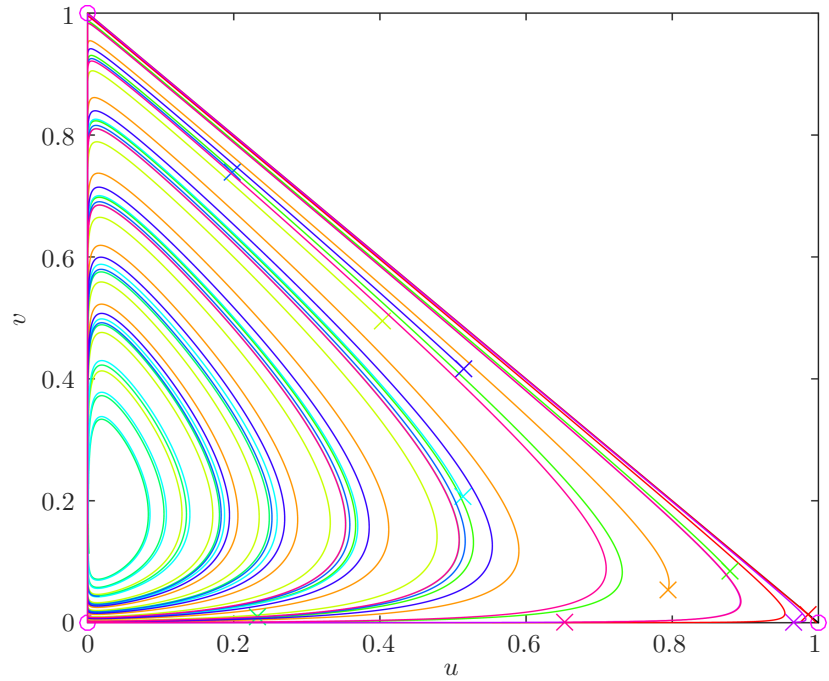}}
        \caption{$a=0.4$, $b=0.1$, $h=0.2$.}
    \end{subfigure}
    \\
    \begin{subfigure}{.9\linewidth}
        \resizebox{\linewidth}{!}{\input{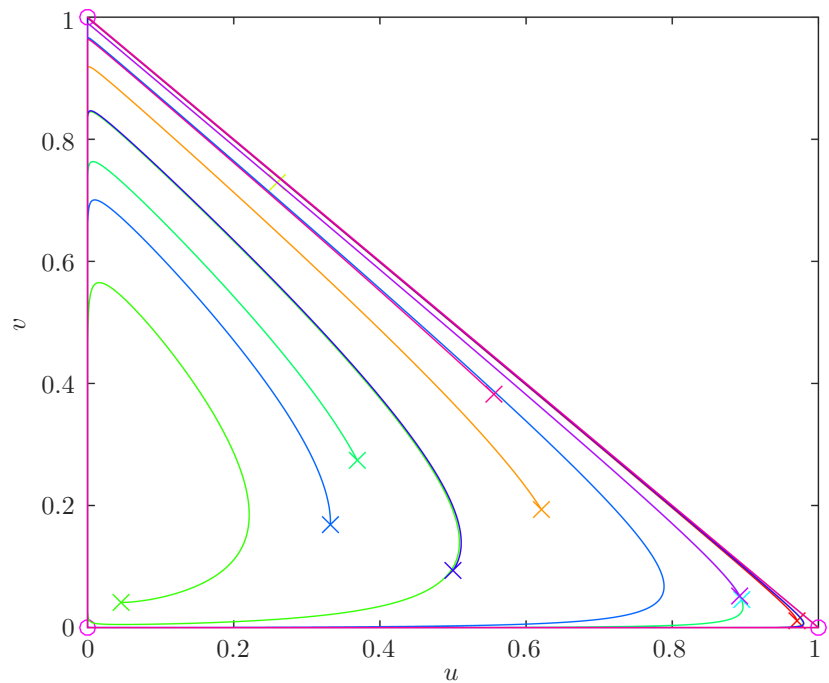}}
        \caption{$a=0.4$, $b=0.1$, $h=0.8$.}
    \end{subfigure}
    \caption{Oscillating trajectories of the homogeneous system with ten random initial conditions 
    and $a<\frac12$. 
    Depending on the value of $h$, the oscillations are either damped (A) or sustained (B). 
    For intermediate values of $h$ (not illustrated here), we find a mixture of both, depending on 
    the initial condition. In all cases,
    the solutions spend a very long time in the neighborhood of $u=0$ (final time of the simulation
    $T=1000$). \textcolor{black}{We refer to Hastings \textit{et al.} \cite{Hastings_2018} for a 
    discussion about the relevance of long transient regimes in biology.}}
    \label{fig:phase_portaits_small_a}
\end{figure}

\subsection{Brake with small selective disadvantage: predator--prey regime\label{subsec:predator-prey_small_b}}
In what follows\textcolor{black}{, we prove that the expected predator--prey structure is indeed true close to $b=0$. In order to do so,} 
it is convenient to understand $b$ as a parameter and to add a subscript $b$ to every object that depends on it. In particular,
we rewrite the term $f_b\left( u,v \right)$ as
\begin{equation}
    f_b\left( u,v \right) =\frac{w_0\left( u,v \right)}{w_b\left( u,v \right)}\left[ f_0\left( u,v \right)+\frac{b}{w_0\left( u,v \right)}\begin{pmatrix}v\left( v+2u+2h\left( 1-u-v \right) \right)\\v\left( v+2u+2h\left( 1-u-v \right)\right)-\left( v+2u+h\left( 1-u-v \right)\right)\end{pmatrix} \right]
    \label{eq:rewritten_reaction}
\end{equation}
with
\begin{align*}
    w_0\left( u,v \right)& =1-au^2-2au\left( 1-u-v \right)\\ & =1+au^2-2au+2auv,
\end{align*}
\begin{align*}
    g_0\left( u,v \right)& =
    \begin{pmatrix}
	\left( 1-a \right)u+2\left( 1-a \right)\left( 1-u-v \right)\\
	v+2u+\left( 1-u-v \right)
    \end{pmatrix}\\
    & =
    \begin{pmatrix}
	2\left( 1-a \right)-\left( 1-a \right)u-2\left( 1-a \right)v\\
	1+u
    \end{pmatrix},
\end{align*}
\begin{align*}
    f_0\left( u,v \right)& =\frac{1}{w_0\left( u,v \right)}g_0\left( u,v \right)-\begin{pmatrix}1\\1\end{pmatrix}\\
    & =\frac{1}{w_0\left( u,v \right)} \begin{pmatrix}-\left( 2a-1 \right)+\left( 3a-1 \right)u-au^2-2v\left( 1-a+au \right)\\u\left( 1+2a -au-2av \right)\end{pmatrix}\\
    & =\frac{1}{w_0\left( u,v \right)} \begin{pmatrix}a\left( 1-u \right)\left( u-\theta \right) -2v\left(1-a+au \right)\\u\left( 1+2a-au-2av \right)\end{pmatrix}.
\end{align*}

\begin{prop}\label{prop:predator-prey}
There exists $\overline{b}_1\in\left(0,1\right]$ such that, if $b\in\left[ 0,\overline{b}_1 \right]$,
then (\ref{sys:general}) has a predator--prey structure where $u$ is the prey and $v$ is 
the predator, namely
\[
    \partial_{v}f_{b,1}<0\text{ and }\partial_{u}f_{b,2}>0\text{ in }\mathsf{T}.
\]
\end{prop}

{\color{black}
\begin{rem*}
The condition $b\leq\overline{b}_1$ is necessary. The predator-prey structure fails if $b$ is too large (see below \subsecref{discussion_fitness}).
\end{rem*}
}

\begin{proof}
    First, we study the case of $b=0$. 

    By straightforward algebra and definition of $\theta$,
    \begin{align*}
	\partial_{v}f_{0,1}\left( u,v \right)& =\frac{2}{w_0\left( u,v \right)^2}\left(-\left( 1-a+au \right)w_0\left( u,v \right)-\left( a\left( 1-u \right)\left( u-\theta \right)-2v\left( 1-a+au \right) \right)au\right)\\
	& =\frac{2}{w_0\left( u,v \right)^2}\left( -2a\left( 2a-1 \right)u^2-\left( 1-a \right) \right) \\
	& \leq -\frac{2}{w_0\left( u,v \right)^2}\min_{z\in[0,1]}\left( 2a\left( 2a-1 \right)z^2+\left(1-a\right) \right).
    \end{align*}

    On one hand, if $a\geq\frac12$, the minimum above is $1-a$. On the other hand, if 
    $a<\frac12$, the minimum above is $-2a\left( 1-2a \right)+ 1-a=4a^2-3a+1>0$.
    In all cases, since $w_0 \leq 1$ in $\mathsf{T}$,
    \[
	\partial_{v}f_{0,1}\left( u,v \right)\leq -\frac{2}{w_0\left( u,v \right)^2}\min\left( 1-a,4a^2-3a+1 \right)\leq-2\min\left( 1-a,4a^2-3a+1 \right)<0.
    \]

    Similarly, 
    \begin{align*}
	\partial_{u}f_{0,2}\left( u,v \right)& =\frac{w_0\left( u,v \right)\left( 1+2a\left( 1-u-v \right) \right)+2au\left( 1+au+2a\left( 1-u-v \right) \right)\left( 1-u-v \right)}{w_0\left( u,v \right)^2}\\
	& \geq \frac{1-a}{w_0\left(u,v\right)^2}\geq 1-a >0.
    \end{align*}
    
    Thus (\ref{sys:general}) in the limiting case $b=0$ has indeed the claimed predator--prey structure.

    Finally, thanks to the smooth convergence of $f_b$ to $f_0$ as $b\to0$ and the preceding 
    uniform estimates, the existence of $\overline{b}_1>0$ as in the statement is immediate.
\end{proof}

\section{Proof of \thmref{convergence_to_0_if_b_is_small}\label{sec:proof}}

In this section, \textcolor{black}{we prove the main result of the paper. In the whole section,} we assume 
$a>\frac{1}{2}$ so the drive has a bistable dynamic and we denote $\theta=\frac{2a-1}{a}$ its intermediate steady state. 

Interestingly, when 
$b\leq \overline{b}_1$, the system under consideration has the structure of a 
predator--prey system with Allee effect on the prey. Such systems have been studied in 
the literature but, apart from numerical simulations that show an incredibly wide and
complicated variety of behaviors, almost nothing is known 
(\textcolor{black}{\textit{e.g.},} \cite{Morozov_Petrovskii_Li,Petrovskii_Mor,Wang_Shi_Wei}). 
Fortunately, the analysis of our particular case is indeed possible, thanks to a 
very simple phase-plane structure (that can easily be observed numerically).
\textcolor{black}{More precisely, the very natural idea of the proof is the following: first, use
phase-plane analysis and Weinberger's maximum principle \cite{Weinberger_1975}
to show that $u$ becomes uniformly smaller than its extinction--persistence threshold $\theta$
in finite time, so that $u$ goes extinct independently of the dynamics of $v$; next, use
the predative structure and scalar comparison arguments to show that when its prey $u$ goes extinct,
the predator $v$ goes extinct as well \footnote{\textcolor{black}{Although this is indeed the main idea, it turns out that technical obstacles arise 
and therefore we will also show in the first step that $v$ becomes uniformly smaller than a constant smaller than $1$.}}.}

We will use as in \subsecref{predator-prey_small_b} the subscripts $b$. In order
to study more precisely the role of $b$ in $f_b$, we define additionally
\begin{align*}
    r(u,v) & =\frac{1}{v}\begin{pmatrix}u\\v\end{pmatrix}\circ\begin{pmatrix}v\left( v\textcolor{black}{+2u}+2h\left( 1-u-v \right) \right)\\v\left( v\textcolor{black}{+2u}+2h\left( 1-u-v \right)\right)-\left( v+2u+h\left( 1-u-v \right)\right)\end{pmatrix}\\
    & = \begin{pmatrix}u\left( v\textcolor{black}{+2u}+2h\left( 1-u-v \right) \right)\\v\left( v\textcolor{black}{+2u}+2h\left( 1-u-v \right)\right)-\left( v+2u+h\left( 1-u-v \right)\right)\end{pmatrix}\\
    & = \begin{pmatrix}u\left( v\textcolor{black}{+2u}+2h\left( 1-u-v \right) \right)\\\left( 1-2h \right)v^2+\left( 3h-1 \right)v-h+2\left(\textcolor{black}{1}-h\right)uv-\left( 2-h \right)u\end{pmatrix}\\
    & = \begin{pmatrix}u\left( v+2u+2h\left( 1-u-v \right) \right)\\-\left( 1-v \right)\left( \left(1-2h\right)v+h \right)-u\left( \textcolor{black}{2\left(1-h\right)\left(1-v\right)+h} \right)\end{pmatrix}.
\end{align*}
Thus the reaction term of (\ref{sys:general}) reads
\begin{equation}
    \begin{pmatrix}u\\v\end{pmatrix}\circ f_b\left( u,v \right) =\frac{w_0\left( u,v \right)}{w_b\left( u,v \right)}\left[ \begin{pmatrix}u\\v\end{pmatrix}\circ f_0\left( u,v \right)+\frac{bv}{w_0\left( u,v \right)}r\left( u,v \right) \right].
    \label{eq:rewritten_reaction_for_contractant_sets}
\end{equation}

\begin{rem*}
    The function $r$ satisfies, in $\mathsf{T}$, $r_1\geq 0$ with equality \textcolor{black}{if and
    only if $u=0$} and $r_2\leq 0$ with equality if and only if $(u,v)=(0,1)$.
    This is mostly obvious but we point out that $(1-2h)v+h =   h(1-v) + v(1-h) \geq 0$ cannot vanish if $v\in (0,1)$.
\end{rem*}

\subsection{Geometric lemmas on the phase-plane structure}
\textcolor{black}{The following lemmas basically state that, when $b$ is small enough, the flow in the interior of $\mathsf{T}$ 
rotates anticlockwise around the repulsive node $\left(\theta, 0\right)$. 
\lemref{contractant_sets_1} is concerned with the region $\left\{u>\theta\right\}$;
\lemref{contractant_sets_2} is concerned with the line segment $\left\{u=\theta\right\}$; \lemref{contractant_sets_3} is concerned
with the region $\left\{u<\theta, v\geq \frac{1+a}{2a}\left(\theta-u\right)\right\}$; \lemref{contractant_sets_4} is concerned with
the instability of $\left(\theta,0\right)$.}

\begin{lem}\label{lem:contractant_sets_1}
There exists $\overline{b}_2\in\left(0,1\right]$ such that, if $b\in\left[ 0,\overline{b}_2 \right]$, 
then for all $\mu\geq0$ and all $u\in\left[ \theta,1 \right]$ such that $\mu\left( u-\theta \right)\leq 1-u$,
\begin{equation}\label{eq:contractance_1}
	\left( \begin{pmatrix}u\\\mu\left( u-\theta \right)\end{pmatrix}\circ f_b\left( u,\mu\left( u-\theta \right) \right) \right)\cdot\begin{pmatrix}\mu\\-1\end{pmatrix}\leq0.
\end{equation}
\end{lem}
\begin{proof}
    Let 
    \[
    \mathsf{C}=\left\{ \left(\mu,u\right)\in\left[0,+\infty\right)\times\left[\theta,1\right]\ |\ \mu\left( u-\theta \right)\leq 1-u \right\}. 
    \]
    In view of (\ref{eq:rewritten_reaction_for_contractant_sets}), (\ref{eq:contractance_1}) is \textcolor{black}{trivial if $\mu=0$ and is
    otherwise} equivalent to
    \[
	\left( \begin{pmatrix}u\\u-\theta\end{pmatrix}\circ f_0\left( u,\mu\left( u-\theta \right) \right)+\frac{b\left( u-\theta \right)}{w_0\left( u,\mu\left( u-\theta \right) \right)}\begin{pmatrix}\mu r_1\left( u,\mu\left( u-\theta \right) \right)\\r_2\left( u,\mu\left( u-\theta \right) \right)\end{pmatrix} \right)\cdot\begin{pmatrix}1\\-1\end{pmatrix}\leq0.
    \]
    Straightforward algebra and the definition of $\theta$ lead to
    \[
	\begin{pmatrix}u\\u-\theta\end{pmatrix}\circ f_0\left( u,\mu\left( u-\theta \right) \right)\cdot\begin{pmatrix}1\\-1\end{pmatrix} = -\frac{u\left( u-\theta \right)}{w_0\left( u,\mu\left( u-\theta \right) \right)}\left(1+a +2a\mu \right)
    \]
    whence (\ref{eq:contractance_1}) is actually equivalent to 
    \[
	-u\left( 1+a+2a\mu \right)+b\begin{pmatrix}\mu r_1\left( u,\mu\left( u-\theta \right) \right)\\r_2\left( u,\mu\left( u-\theta \right) \right)\end{pmatrix}\cdot\begin{pmatrix}1\\-1\end{pmatrix}\leq0.
    \]
    Clearly, this is true for all $\mu\geq 0$ and all $u\in\left[\theta,1\right]$ if $b=0$. Hence from now 
    on we focus on the case $b>0$. In such a case, requiring (\ref{eq:contractance_1}) for all 
    $\left(\mu,u\right)\in\mathsf{C}$ is equivalent to requiring
    \[
	\frac{1}{b}\geq\sup_{\left(\mu,u\right)\in\mathsf{C}}\frac{\max\left(0,\begin{pmatrix}\mu r_1\left( u,\mu\left( u-\theta \right) \right)\\r_2\left( u,\mu\left( u-\theta \right) \right)\end{pmatrix}\cdot\begin{pmatrix}1\\-1\end{pmatrix}\right)}{u\left( 1+a+2a\mu \right)}.
    \]
    In view of the signs of $r_1$ and $r_2$ in $\mathsf{T}$, the right-hand side 
    is positive, so that the above inequality reduces to 
    \[
	\frac{1}{b}\geq\sup_{\left(\mu,u\right)\in\mathsf{C}}\frac{\begin{pmatrix}\mu r_1\left( u,\mu\left( u-\theta \right) \right)\\r_2\left( u,\mu\left( u-\theta \right) \right)\end{pmatrix}\cdot\begin{pmatrix}1\\-1\end{pmatrix}}{u\left( 1+a+2a\mu \right)}>0.
    \]
    It only remains to verify the finiteness of this supremum. 

    Let $\left(\mu,u\right)\in\mathsf{C}$. From
    \[
	\mu r_1\left( u,\mu\left( u-\theta \right) \right)=\left( \mu\textcolor{black}{+2}-2h\left( \mu+1 \right) \right)\mu u^2+\left( 2h+2h\mu\theta-\mu\theta \right)\mu u
    \]
    and
    \begin{align*}
	-r_2\left( u,\mu\left( u-\theta \right) \right) = &-\left( \mu\textcolor{black}{+2}-2h\left( \mu+1 \right) \right)\mu u^2\\
	& -\left( \mu\left( 4h\mu\theta+3h+2h\theta-1-2\mu\theta\textcolor{black}{-2\theta} \right)+h-2 \right)u\\
	& -\left( \mu^2\theta^2+\mu\theta-2h\mu^2\theta^2-3h\mu\theta-h \right),
    \end{align*}
    it follows
    \begin{align*}
	\begin{pmatrix}\mu r_1\left( u,\mu\left( u-\theta \right) \right)\\r_2\left( u,\mu\left( u-\theta \right) \right)\end{pmatrix}\cdot\begin{pmatrix}1\\-1\end{pmatrix}= &\left( (1-2h)\mu^2\theta-(h+2h\theta-1\textcolor{black}{-2\theta})\mu+2-h \right)u\\
	&-\left( (1-2h)\mu^2\theta^2+(1-3h)\mu\theta-h \right),
    \end{align*}
    which reads, as a polynomial of $\mu$,
    \[
	\begin{pmatrix}\mu r_1\left( u,\mu\left( u-\theta \right) \right)\\r_2\left( u,\mu\left( u-\theta \right) \right)\end{pmatrix}\cdot\begin{pmatrix}1\\-1\end{pmatrix}= \left( 1-2h \right)\theta\left( u-\theta \right)\mu^2+\left( \left( 3h-1 \right)\theta-\left( 2h\theta+h-1\textcolor{black}{-2\theta} \right)u \right)\mu+\left( 2-h \right)u+h.
    \]

    On one hand, if $h>\frac12$, the above second-order polynomial of $\mu$ is bounded above by 
    some constant, whence 
    \[
	\sup_{\left(\mu,u\right)\in\mathsf{C}}\frac{\begin{pmatrix}\mu r_1\left( u,\mu\left( u-\theta \right) \right)\\r_2\left( u,\mu\left( u-\theta \right) \right)\end{pmatrix}\cdot\begin{pmatrix}1\\-1\end{pmatrix}}{u\left( 1+a+2a\mu \right)}<+\infty.
    \]
    
    On the other hand, if $h\leq\frac12$, then we use $\mu\left(u-\theta\right)\leq 1-u$ and obtain
    \begin{align*}
	\begin{pmatrix}\mu r_1\left( u,\mu\left( u-\theta \right) \right)\\r_2\left( u,\mu\left( u-\theta \right) \right)\end{pmatrix}\cdot\begin{pmatrix}1\\-1\end{pmatrix} & \leq \left( 1-2h \right)\theta\mu\left(1-u\right)+\left( \left( 3h-1 \right)\theta-\left( 2h\theta+h-1\textcolor{black}{-2\theta} \right)u \right)\mu+\left( 2-h \right)u+h \\
	& \leq \left(\left(1\textcolor{black}{+}\theta\right)u-h\left(u-\theta\right)\right)\mu+\left( 2-h \right)u+h,
    \end{align*}
    whence 
    \[
	\sup_{\left(\mu,u\right)\in\mathsf{C}}\frac{\begin{pmatrix}\mu r_1\left( u,\mu\left( u-\theta \right) \right)\\r_2\left( u,\mu\left( u-\theta \right) \right)\end{pmatrix}\cdot\begin{pmatrix}1\\-1\end{pmatrix}}{u\left( 1+a+2a\mu \right)}<+\infty
    \]
    is again true.

    The proof is ended with 
    \[
	\overline{b}_2=\min\left( 1,\left(  \sup_{\left(\mu,u\right)\in\mathsf{C}}\frac{\begin{pmatrix}\mu r_1\left( u,\mu\left( u-\theta \right) \right)\\r_2\left( u,\mu\left( u-\theta \right) \right)\end{pmatrix}\cdot\begin{pmatrix}1\\-1\end{pmatrix}}{u\left( 1+a+2a\mu \right)}\right)^{-1} \right).
    \]
\end{proof}

\textcolor{black}{\lemref{contractant_sets_2} and \lemref{contractant_sets_3} below are proved by similar (and simpler) considerations.
For the sake of brevity, we do not detail the proofs. Note that \lemref{contractant_sets_3} holds true for any admissible value of $b$.}

\begin{lem}\label{lem:contractant_sets_2}
There exists $\overline{b}_3\in\left(0,1\right]$ such that, if $b\in\left[ 0,\overline{b}_3 \right]$, 
then for all $v\in\left[0,1-\theta\right]$,
\begin{equation*}
	\left( \begin{pmatrix}\theta\\v\end{pmatrix}\circ f_b\left( \theta,v \right) \right)\cdot\begin{pmatrix}1\\0\end{pmatrix}\leq0.
\end{equation*}
\end{lem}

\textcolor{black}{\begin{lem}\label{lem:contractant_sets_3}
For all $\mu\leq-\frac{1+a}{2a}$ and all $u\in\left[ 0,\theta \right]$ such that $-\mu\left( \theta-u \right)\leq 1-u$,
\begin{equation*}
	\left( \begin{pmatrix}u\\\mu\left( u-\theta \right)\end{pmatrix}\circ f_b\left( u,\mu\left( u-\theta \right) \right) \right)\cdot\begin{pmatrix}\mu\\-1\end{pmatrix}\geq0.
\end{equation*}
\end{lem}}

Next, using the fact that $(\theta,0)$ is an unstable node (see \subsecref{linear_stability}), 
we can prove the following similar lemma.

\begin{lem}\label{lem:contractant_sets_4}
There exists $\overline{\eta}\in\left(0,\frac{1-\theta}{4}\right)$ such that, for all
$\eta\in\left(0,\overline{\eta}\right)$ and all $\mu>0$:
\begin{enumerate}
    \item for all $u\in\left[ \theta-\eta,\theta+\eta \right]$,
    \[
	\left( \begin{pmatrix}u\\\min\left(1,\frac{\mu}{2}\right)\left( u-\theta+\eta \right)\end{pmatrix}\circ f_b\left( u,\min\left(1,\frac{\mu}{2}\right)\left( u-\theta+\eta \right) \right) \right)\cdot\begin{pmatrix}\min\left(1,\frac{\mu}{2}\right)\\-1\end{pmatrix}\leq0.
	\]
    \item the straight lines of equation $v=\mu\left(u-\theta\right)$ and 
    $v=\min\left(1,\frac{\mu}{2}\right)\left(u-\theta+\eta \right)$ intersect \textcolor{black}{at 
    \[
	u=
	\begin{cases}
	    \theta+\eta & \text{if }\mu\leq 2, \\
	    \theta+\frac{\eta}{\mu-1} & \text{if }\mu>2,
	\end{cases}
    \]
    and the intersection point is in the interior of $\mathsf{T}$.}
\end{enumerate}
\end{lem}

These lemmas will be used to construct a family of convex sets, illustrated in \figref{sets_C_mu}.

\subsection{Main proof}
We are now in position to prove the main theorem\textcolor{black}{, \thmref{convergence_to_0_if_b_is_small}, whose statement is recalled and precised below}.

\begin{thm*}
    Assume $b\leq\min\left(\overline{b}_1,\overline{b}_2,\overline{b}_3\right)$.

    Let $(u,v)$ be the solution of (\ref{sys:general}) with initial condition 
    $\left(u_0,v_0\right)\in\mathscr{C}\left(\mathbb{R}^N,\mathsf{T}\right)$ satisfying
    \[
	v_0\neq 0\quad\text{and}\quad\lim_{\|x\|\to+\infty}\left( u_0,v_0 \right)(x)=\left( 0,0 \right).
    \]
    
    Then coextinction occurs, namely
    \[
	\lim_{t\to+\infty}\left(\sup_{x\in\mathbb{R}^N}u\left( t,x \right)+\sup_{x\in\mathbb{R}^N}v\left( t,x \right)\right)=0.
    \]
\end{thm*}
\begin{proof}
    First, for all $\mu\geq0$, we define the line segment
    \[
	\mathsf{S}_\mu = \left\{ \left( \tilde{u},\tilde{v} \right)\in\mathsf{T}\ |\ \tilde{v}=\mu\left( \tilde{u}-\theta \right) \right\}.
    \]
    For all $t\geq0$, we define the \textcolor{black}{spatial image}
    \[
	\textcolor{black}{\mathsf{I}}_t = \left\{ \left( u,v \right)\left( t,x \right)\ |\ x\in\mathbb{R}^N \right\}\subset\mathsf{T}.
    \]
    For all $t\geq0$ and $\mu\geq0$, we define the (Euclidean \textcolor{black}{in $\mathbb{R}^2$}) distance between $\textcolor{black}{\mathsf{I}}_t$ and $\mathsf{S}_\mu$
    \[
	d\left( t,\mu \right)=\text{dist}\left(\textcolor{black}{\mathsf{I}}_t,\mathsf{S}_\mu\right).
    \]

    By standard parabolic estimates and the fact that $(0,0)$ is a solution of 
    (\ref{sys:general}), the assumptions imply 
    \[
	\lim_{\|x\|\to\infty}\left( u,v \right)(t,x)=\left( 0,0 \right)\text{ for all }t\geq0.
    \]
    Consequently, for all $t\geq0$ and $\mu\geq0$, there exists $x_{t,\mu}\in\mathbb{R}^N$ such that
    \[
	d(t,\mu) =\text{dist}\left( \left( u,v \right)\left( t,x_{t,\mu} \right),\mathsf{S}_\mu\right).
    \]
    Now, let 
    \[
    T=\sup\left\{t\geq0\ |\ \textcolor{black}{\mathsf{I}}_t\cap\left( \theta,1 \right]\times\left[ 0,1 \right]\neq\emptyset\right\}
    \]
    and let us show that $T<+\infty$. Without loss of generality, we restrict the analysis to the case
    $T>0$. 

    For all $t\in\left[ 0,T \right]$, we can define 
    \[
	\mu_t=\inf\left\{ \mu\geq0\ |\ d\left( t,\mu \right)=0\right\}
    \]
    and, by continuity, we find
    \[
	d\left( t,\mu_t \right)=\text{dist}\left( \left( u,v \right)\left( t,x_{t,\mu_t} \right),
	\mathsf{S}_{\mu_t}\right)=0.
    \]

    Next, by virtue of Weinberger's maximum principle \cite{Weinberger_1975} applied in 
    the convex invariant set     $\mathsf{T}$ satisfying the so-called slab condition, for all $t\in\left(0,T\right]$,
    $\left( u,v \right)\left( t,x_{t,\mu_t} \right)\in\text{int}\left(\mathsf{T}\right)$. This directly
    implies that for all $t\in\left(0,T\right]$, $\mu_t>0$. 

    However, in view of \lemref{contractant_sets_1} and \lemref{contractant_sets_4}, for any 
    $\eta\in\left(0,\overline{\eta}\right)$ and any $\mu>0$, the convex set satisfying a slab condition
    \[
	\mathsf{C}_{\mu} = \left\{ \left( \tilde{u},\tilde{v} \right)\in\mathsf{T}\ |\ \tilde{v}\geq\mu\left( \tilde{u}-\theta \right),\tilde{v}\geq \min\left(1,\frac{\mu}{2}\right)\left( \tilde{u}-\theta+\eta \right) \right\}
    \]
    is again an invariant set for (\ref{sys:general}). Now we define $T_0=\min\left(1,\frac{T}{2}\right)$ 
    and fix $\eta\in\left(0,\overline{\eta}\right)$ so small that 
    \[
    \textcolor{black}{\mathsf{I}}_{T_0}\subset\mathsf{C}_{\mu_{T_0}}.
    \]
    Applying again Weinberger's maximum principle,
    we find that for all $\left( t,t' \right)\in\left[ T_0,T \right]^2$ such that $t'>t>T_0$, 
    $\left( u,v \right)\left(t',x_{t',\mu_{t'}}\right)\in\text{int}\left( \mathsf{C}_{\mu_t} \right)$. 
    Hence the family $\left(\mathsf{C}_{\mu_t}\right)_{t\geq T_0}$ is decreasing
    with respect to the inclusion, or in other words the family $\left( \mu_t \right)_{t\geq T_0}$ 
    is increasing. 
    
\begin{figure}
    \resizebox{.9\linewidth}{!}{\input{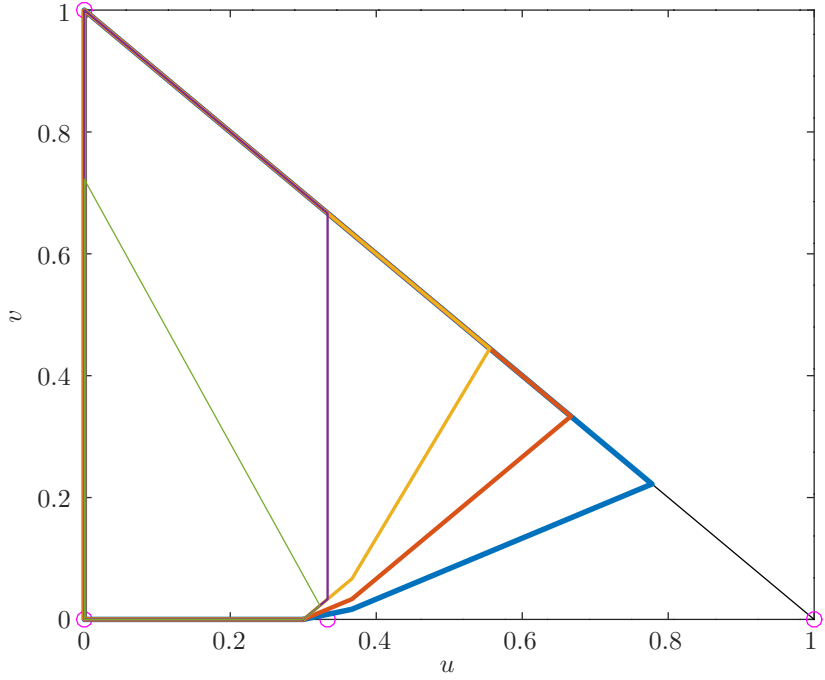}}
    \caption{\textcolor{black}{Each colored polygon is the boundary of a set $\mathsf{C}_\mu$  
    (from right to left: $\mu=0.5, 1, 2, +\infty, -\frac12\left(\frac{1+a}{2a}+\frac{1}{\theta}\right)$).\\
    (The value $\eta=\frac{\theta}{10}$ was chosen here for graphic clarity and was not analytically or numerically checked.)}}
    \label{fig:sets_C_mu}
\end{figure}    

    Assuming by contradiction that $T=+\infty$, we can define 
    \[
	\mu_\infty=\lim_{t\to +\infty}\mu_t\in\left( \mu_{T_0},+\infty \right].
    \]
    \textcolor{black}{We extend naturally the definition of $\mathsf{C}_\mu$ and $\mathsf{S}_\mu$ to the case $\mu=\infty$:}
    \[
    \mathsf{C}_{\infty} = \left\{ \left( \tilde{u},\tilde{v} \right)\in\mathsf{T}\ |\ \tilde{u}\leq\theta,\tilde{v}\geq\left( \tilde{u}-\theta+\eta \right) \right\}
    \quad\text{and}\quad
    \mathsf{S}_\infty =\left\{ \theta \right\}\times\left[ 0,1-\theta \right].
    \]
    Then, by standard parabolic estimates, the family 
    \[
	\left(\left( t,x \right)\mapsto\left( u,v \right)\left( t+n,x+x_{n,\mu_n} \right)\right)_{n\in\mathbb{N}}
    \]
    converges locally uniformly up to extraction to an entire solution 
    $\left( u_\infty,v_\infty \right)$ of (\ref{sys:general}) valued in $\mathsf{C}_{\mu_\infty}$
    and satisfying
    $\left( u_\infty,v_\infty \right)\left( 0,0 \right)\in\mathsf{S}_{\mu_\infty}$.
    Applying again Weinberger's maximum principle and using the fact that 
    $\left(\theta,0\right)\not\in\mathsf{C}_{\mu_\infty}$, 
    we find that necessarily $\mu_\infty=+\infty$. However, by the scalar comparison principle and 
    the predator-prey structure (see \propref{predator-prey}),
    $u_\infty$ is a subsolution for the bistable equation 
    \[
	\mathscr{P}\tilde{u} = \tilde{u} f_{b,1}\left( \tilde{u},0 \right)
    \]
    satisfying $u_\infty\leq\theta$ with $u_\infty(0,0)=\theta$. This directly implies 
    $u_\infty=\theta$, and then back to the system $v_\infty=0$, which \textcolor{black}{contradicts
    $\left( u_\infty,v_\infty \right)\in\mathsf{C}_{\infty}$.}

    Therefore $T<+\infty$, that is $(u,v)$ enters in finite time the domain $\mathsf{C}_\infty$. 
    By \lemref{contractant_sets_2} and again by Weinberger's maximum principle applied this time to $\mathsf{C}_{\infty}$, 
    it actually enters in finite time the interior of $\mathsf{C}_\infty$. 
    {\color{black} By \lemref{contractant_sets_3} and again by Weinberger's maximum principle, 
    it actually enters in finite time the interior of the convex set
    \[
    \mathsf{C}_{-\frac12\left(\frac{1+a}{2a}+\frac{1}{\theta}\right)}=\left\{ \left(\tilde{u},\tilde{v}\right)\in\mathsf{T}\ |\ \tilde{v}\leq-\frac12\left(\frac{1+a}{2a}+\frac{1}{\theta}\right)\left(\tilde{u}-\theta\right),\tilde{v}\geq\left( \tilde{u}-\theta+\eta \right)\right\}.
    \]
    Here we use the inequality $\theta^{-1}>\frac{1+a}{2a}$ which implies the inequality
    $-\frac12\left(\frac{1+a}{2a}+\frac{1}{\theta}\right)<-\frac{1+a}{2a}$.
    
    Therefore we have in finite time, say $\tilde{T}\geq T$, both $u<\theta$ and 
    \[
    v<\frac{\theta}{2}\left(\frac{1+a}{2a}+\frac{1}{\theta}\right)=\frac{4a^2-\left(1-a\right)}{4a^2}<1.
    \]
    }
    Repeating the comparison with the scalar bistable equation 
    \[
	\mathscr{P}\tilde{u} =\tilde{u} f_{b,1}\left( \tilde{u},0 \right),
    \]
    \textcolor{black}{with this time $x\mapsto u\left(\tilde{T},x\right)$ as initial condition,}
    it follows that $u\to 0$ uniformly in space asymptotically in time. 
    
    In view of the convergence of $u$, the function
    \[
	\varepsilon:t\mapsto \sup_{t'\geq t}\left ( \max_{x\in\mathbb{R}^N}u(t',x)\right )
    \]
    converges monotonically to $0$.
    Thanks to the predator--prey structure, it follows just as before that, for all $t_0\geq0$, $v$ is
    in $\left[t_0,+\infty\right)\times\mathbb{R}^N$ a subsolution for the equation
    \[
	\mathscr{P}\tilde{v} = \tilde{v} f_{b,2}\left(\varepsilon\left(t_0\right),\tilde{v}\right).
    \]
    The Taylor expansion of $f_{b,2}$ and the predator--prey structure bring forth the 
    existence of a positive constant $C>0$ such that\textcolor{black}{, for large values of $t_0$,}
    \begin{align*}
	\mathscr{P}v & \leq v\left[ \frac{b\left( -h+\left( 2h-1 \right)v \right)\left( 1-v \right)}{1-2hbv+\left( 2h-1 \right)bv^2} + \partial_{u}f_{b,2}(0,v)\varepsilon\left(t_0\right) + o\left( \varepsilon\left(t_0\right) \right) \right]\\
	& \leq v\left[ \frac{b\left( -h+\left( 2h-1 \right)v \right)\left( 1-v \right)}{1-2hbv+\left( 2h-1 \right)bv^2} + C \varepsilon\left(t_0\right) \right].
    \end{align*}
    Now, thanks to $w_b\leq 1$ and $-h+\left( 2h-1 \right)v\leq 0$, 
    \[
	\frac{b\left( -h+\left( 2h-1 \right)v \right)\left( 1-v \right)}{1-2hbv+\left( 2h-1 \right)bv^2}\leq-b\left(h-(2h-1)v \right)\left( 1-v \right).
    \]
    If $h=1$, we find 
    \[
	\frac{b\left( -h+\left( 2h-1 \right)v \right)\left( 1-v \right)}{1-2hbv+\left( 2h-1 \right)bv^2} \leq -b\left( 1-v \right)^2.
    \]
    If $h\in\left(\frac12,1\right)$, we find
    \[
	\frac{b\left( -h+\left( 2h-1 \right)v \right)\left( 1-v \right)}{1-2hbv+\left( 2h-1 \right)bv^2} \leq -b\left( 1-h \right)\left( 1-v \right).
    \]
    If $h\in\left(0,\frac12\right)$, we find
    \[
	\frac{b\left( -h+\left( 2h-1 \right)v \right)\left( 1-v \right)}{1-2hbv+\left( 2h-1 \right)bv^2} \leq -bh\left( 1-v \right).
    \]
    If $h=0$, we find
    \[
	\frac{b\left( -h+\left( 2h-1 \right)v \right)\left( 1-v \right)}{1-2hbv+\left( 2h-1 \right)bv^2} \leq -bv\left( 1-v \right).
    \]
    Hence the function $\hat{v}=1-v$ satisfies
    \[
	\mathscr{P}\hat{v} \geq \left( 1-\hat{v} \right)\left( F\left( \hat{v} \right)-C \varepsilon\left( t_0 \right) \right)\textcolor{black}{\text{ in }\left[t_0,+\infty\right)\times\mathbb{R}^N,}
    \]
    where 
    \[
	F:z\in\left[ 0,1 \right]\mapsto
	\begin{cases}
	    bz^2 & \text{if }h=1,\\
	    b\left( 1-h \right)z & \text{if }h\in\left(\frac12,1\right)\\
	    bhz & \text{if }h\in\left(0,\frac12\right)\\
	    bz\left(1-z\right) & \text{if }h=0.
	\end{cases}
    \]
    In all cases, $F$ is increasing in $\left(0,\frac12\right)$ and maps
    $\left[0,\frac12\right]$ onto $\left[0,F\left(\frac12\right)\right]$. Assume that
    $t_0$ is so large that $C\varepsilon\left(t_0\right)<F\left(\frac12\right)$ and let
    $z_0\in\left(0,\frac12\right)$ be the root of $F(z)=C\varepsilon(t_0)>0$. 
    \textcolor{black}{Assume further that $t_0$ is so large that $t_0\geq\tilde{T}$ and $z_0<1-\frac{4a^2-\left(1-a\right)}{4a^2}$, so that
    \[
    \min_{x\in\mathbb{R}^N}\hat{v}\left(t_0,x\right)>z_0.
    \]
    }
    
    To conclude, we distinguish two cases, depending on whether $h=0$ or not. 
    
    {\color{black}
    On one hand, if $h>0$, the reaction term 
    $\left( 1-\hat{v} \right)\left( F\left( \hat{v} \right)-C\varepsilon\left( t_0 \right) \right)$
    has in $\left[z_0,1\right]$ a monostable structure: its only zeros are $z_0$ and $1$ with
    negative derivative at $1$ and positive derivative at $z_0$. 
    Again by standard comparison, this implies that the solution of
    \[
	\begin{cases}
	    \mathscr{P}z = \left( 1-z \right)\left( F\left( z \right)-C \varepsilon\left( t_0 \right) \right)\\
	    z(t_0,x)=\hat{v}\left( t_0,x \right)>z_0\text{ for all }x\in\mathbb{R}
	\end{cases}
    \]
    converges uniformly in space to $1$ as $t\to+\infty$. Finally, by the comparison principle,
    $\hat{v}$ converges uniformly to $1$, or in other words $v$ converges uniformly to $0$.
    
    On the other hand, if $h=0$, then $\left( 1-\hat{v} \right)\left( F\left( \hat{v} \right)-C\varepsilon\left( t_0 \right) \right)$
    has in $\left(z_0,1\right)$ one more zero, say $z_1$. The steady
    states $z_0$ and $1$ are unstable whereas $z_1$ is stable. }By arguments similar to the previous
    case, we obtain 
    \[
        \lim_{t\to+\infty}\inf_{x\in\mathbb{R}^N} \hat{v}(t,x)\geq z_1.
    \]
    But $z_1$ depends continuously on $t_0$ and converges to $1$ as $t_0\to+\infty$. Passing to the limit,
    we deduce
    \[
        \lim_{t\to+\infty}\inf_{x\in\mathbb{R}^N} \hat{v}(t,x)\geq 1,
    \]
    whence $v$ converges uniformly to $0$ indeed.
\end{proof}

\section{Discussion}
\textcolor{black}{In this last section, we discuss the results and suggest new directions of research. More precisely,
we discuss different parameter ranges (\subsecref{discussion_fitness}, \subsecref{discussion_efficiency}),
a different PDE model (\subsecref{discussion_Nagylaki}) and biological implications (\subsecref{discussion_bio}) }

\subsection{On different choices of selective disadvantages\label{subsec:discussion_fitness}}
In what follows, we give some evidence arguing  in favor of coextinction when $a>\frac12$ (and $b\leq a$),
versus the existence of a joint front with complex spatio-temporal dynamics when $a<\frac12$, 
see also \figref{behaviours}.

\subsubsection{Monostable drive and brake with small selective disadvantage}
In this subsection we assume $a\leq\frac12$ so that the drive has a monostable dynamic and 
$b\leq\overline{b}_1$ so that we are in the predator--prey regime (see \propref{predator-prey}). 
We recall that, by \subsecref{linear_stability}, global convergence to $(0,0)$ for the diffusionless system
cannot occur if $a<\frac12$, as it is a saddle. As showed previously by \figref{phase_portaits_small_a}, 
the system is oscillating with either damped oscillations eventually converging to a coexistence state
or sustained oscillations approaching a periodic limit cycle or the heteroclinic cycle on the boundary. 
What about the reaction--diffusion system? In particular, is the monostable drive fast enough to persist 
ahead of the oscillating front?

First, since $v\mapsto f_2(1,v)$ is
\[
v\mapsto \frac{2(1-b)-b(1-h)v}{1-a+(2h-1)bv^2+2(a-b)v}-1
\]
and is therefore decreasing with respect to $v$, the nonlinearity $v f_2(1,v)$ has a KPP structure.

Hence (\ref{sys:general}) is a particular case of the more general framework investigated by Ducrot--Giletti--Matano
\cite{Ducrot_Giletti_Matano_2}, provided that $a\leq\frac{1}{4}$ (the prey follows KPP-like dynamics), and the drift term due to a heterogeneous population size is neglected, \textit{i.e.} 
$\mathscr{P}=\partial_t -\Delta$. However, 
regarding the part of their analysis we are interested in here, we believe the KPP condition and the absence
of drift could be relaxed. In any case, the following discussion can be rigorously justified when $a\leq\frac{1}{4}$ and $n$ is spatially homogeneous following the arguments in \cite{Ducrot_Giletti_Matano_2}. 

This boils down to the comparison between the spreading speed $c$ of $u$ invading the wild-type type $0$ alone (cf. \subsecref{without_brake}) and the spreading speed of $v$ invading the gene drive $u\sim 1$  separately, namely $2\sqrt{f_2(1,0)}=2\sqrt{\frac{1+a-2b}{1-a}}$.

In the case $a\leq \frac{1}{4}$, the sole gene drive verifies a KPP condition ensuring that the propagation occurs at the explicit speed $2\sqrt{f_1(0,0)}=2\sqrt{1-2a}$. It is immediate to see that $2\sqrt{1-2a} < 2\sqrt{\frac{1+a-2b}{1-a}}$, so that $u$ is unable to evade, and will eventually be caught up by $v$.

In the case $a\in\left( \frac{1}{4},\frac12\right)$, the spreading speed $c$ can be estimated by comparison with a well-chosen KPP equation, namely with a KPP
equation whose reaction term is larger than or equal to $u\mapsto uf_1(u,0)$ in $\left[0,1\right]$.
It can be verified easily that the minimal KPP reaction term satisfying this condition is exactly defined in
$\left[0,1\right]$ as
\[
u\mapsto u \max_{z\in\left[u,1\right]}f_1(z,0).
\]
Denoting 
\[
\alpha=\max_{u\in\left[0,1\right]}f_1(u,0),
\]
the KPP speed associated with this reaction term is exactly $2\sqrt{\alpha}$, so that the spreading speed $c$ 
of $u$ invading $0$ satisfies $c\leq 2\sqrt{\alpha}$. To estimate $\alpha$ itself, \textcolor{black}{we use the fact 
that, since $u\mapsto f_1(u,0)$ is not decreasing, $\alpha$ is exactly such that the 
second-order polynomial equation in the variable $u$
\[
\alpha\left( 1-a+a\left( 1-u \right)^2 \right)=a\left( 1-u \right)\left( u-\theta \right)
\]
admits a positive double root.}
After some algebra, we find that the corresponding discriminant vanishes if and only if
$\alpha = \frac{1-\sqrt{a}}{2\sqrt{a}}$, whence the estimate $c\leq 2\sqrt{\alpha}$ finally reads
\[
    c\leq\sqrt{2\frac{1-\sqrt{a}}{\sqrt{a}}}.
\]

Subsequently, we consider the quantity
\[
2\sqrt{\frac{1+a-2b}{1-a}}-\sqrt{2\frac{1-\sqrt{a}}{\sqrt{a}}}.
\]
It turns out that there exists $a_{1,b}\in\left(0,\frac12\right)$, increasing with respect to $b\in[0,a]$
and satisfying $a_{1,a}=\frac{1}{9}$, such that this 
quantity is positive if and only if $a> a_{1,b}$. Consequently, if $\frac{1}{2}\geq a\geq \frac{1}{4}>a_{1,b}$, 
we expect that $u$ is never able to evade $v$ and a joint invasion front is observed. 

Due to the diffusion and possible instabilities, we 
cannot really expect simple sustained or damped oscillations in the wake of the joint invasion front.
Numerically, we observe very complicated spatio-temporal patterns (see \figref{Cauchy_small_a}). 

\begin{figure}
    \centering
        \begin{subfigure}{0.45\linewidth}
    \includegraphics[width=\linewidth]{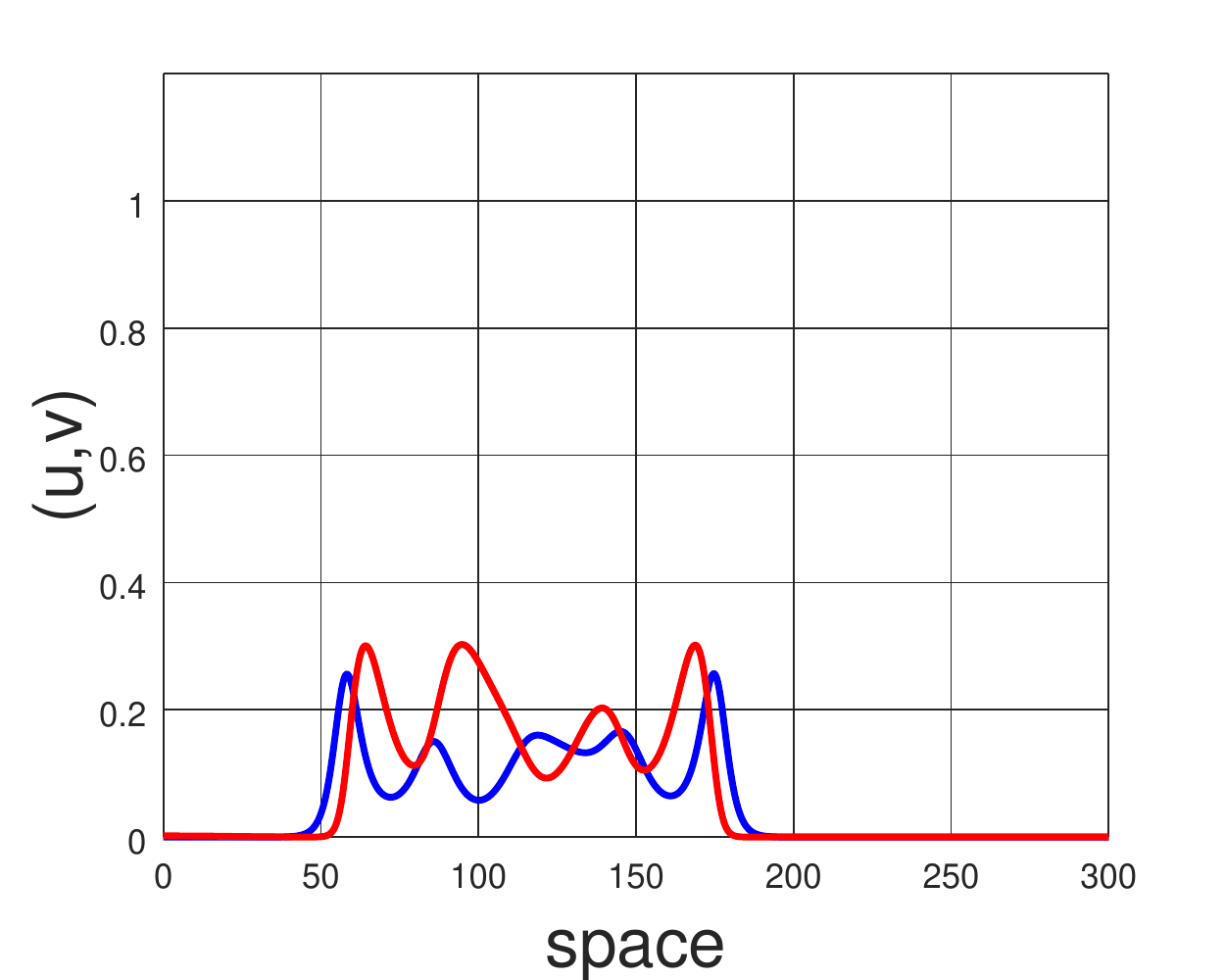}
    \caption{$t = 240$}
    \end{subfigure}
    \begin{subfigure}{0.45\linewidth}
    \includegraphics[width=\linewidth]{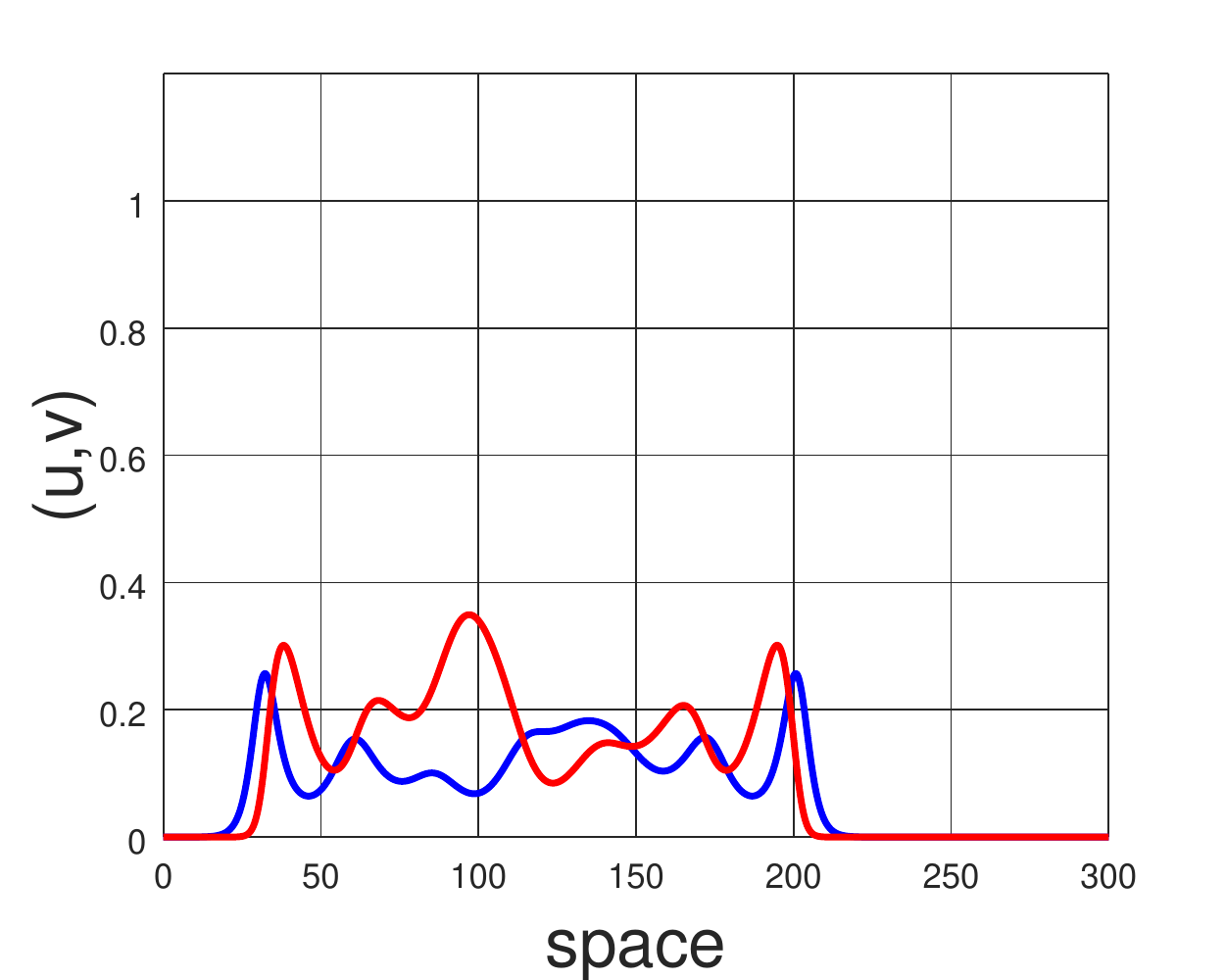}
    \caption{$t = 280$}
    \end{subfigure}
    \begin{subfigure}{0.45\linewidth}
    \includegraphics[width=\linewidth]{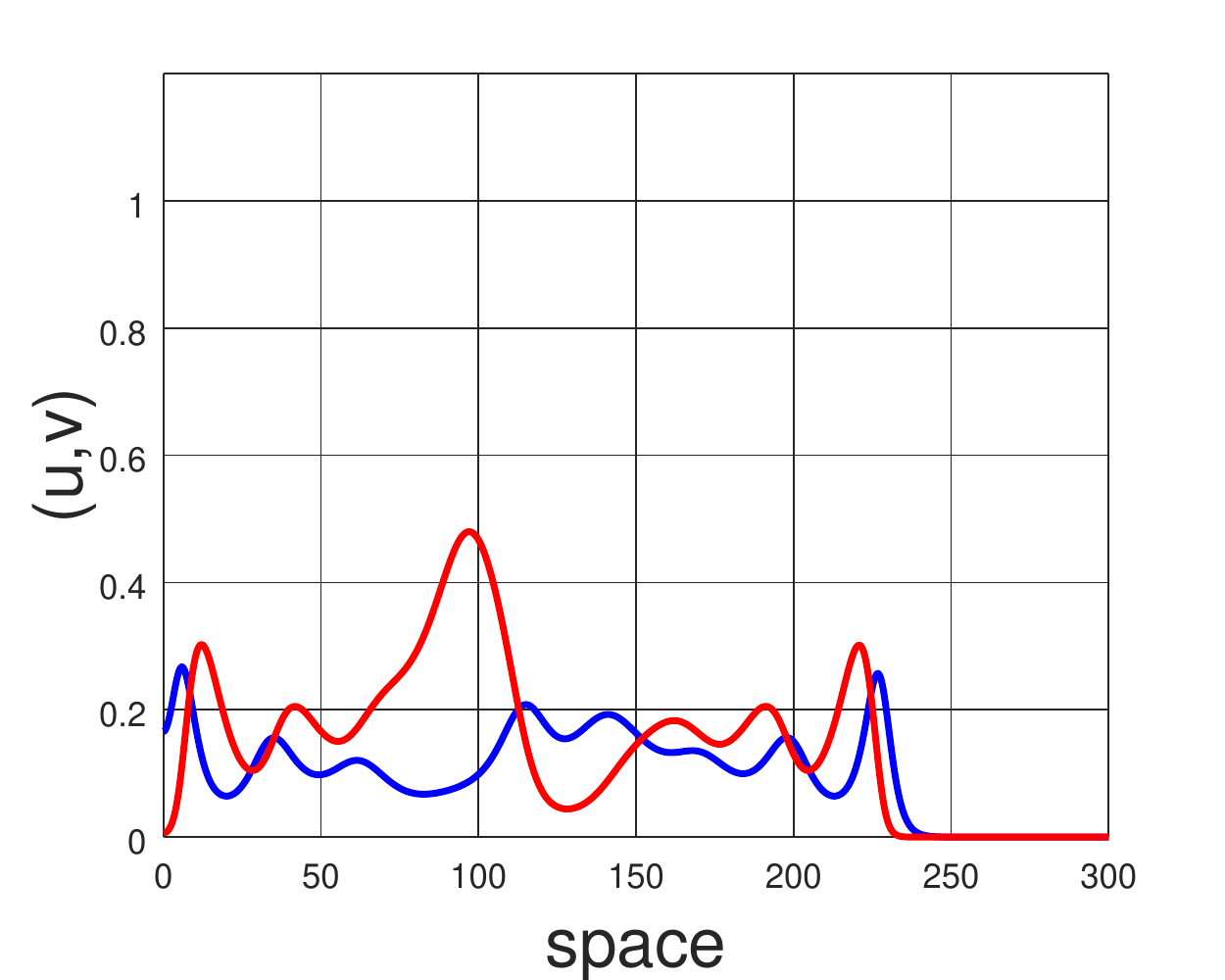}
    \caption{$t = 320$}
    \end{subfigure}
    \begin{subfigure}{0.45\linewidth}
    \includegraphics[width=\linewidth]{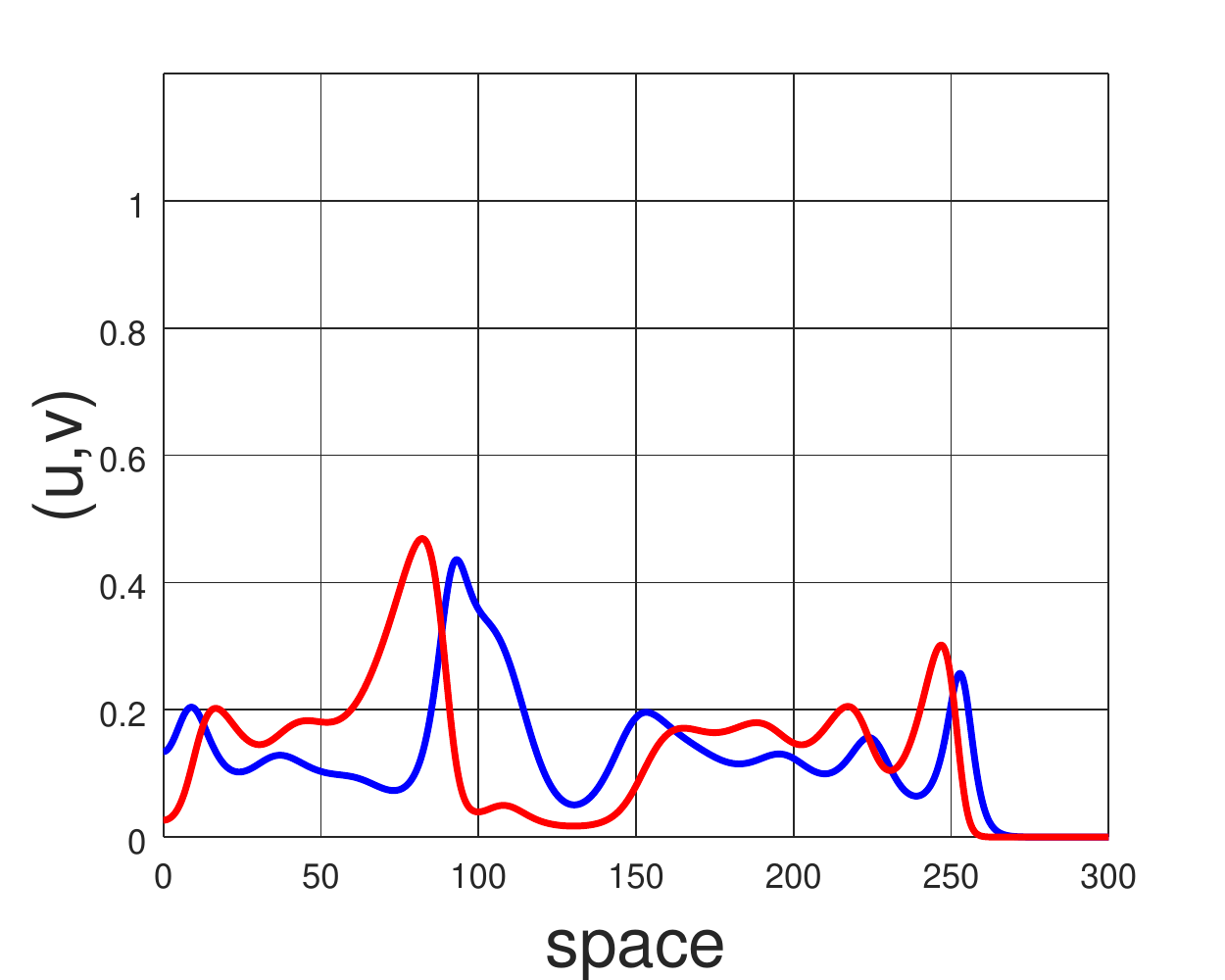}
    \caption{$t = 360$}
    \end{subfigure}
    \caption{Further numerical snapshots of the solution of (\ref{sys:general}) in the case $a= 0.45$, and $b = 0.35$, $h = 0.5$, and $n$ is spatially homogeneous, see \figref{behavioursA}.}
    \label{fig:Cauchy_small_a}
\end{figure}

%

\subsubsection{Brake and drive with the same large selective disadvantage}
When $b$ is close to $a>\frac12$, the dynamics are more complicated. 

\begin{figure}
    \begin{subfigure}{.9\linewidth}
        \resizebox{\linewidth}{!}{\input{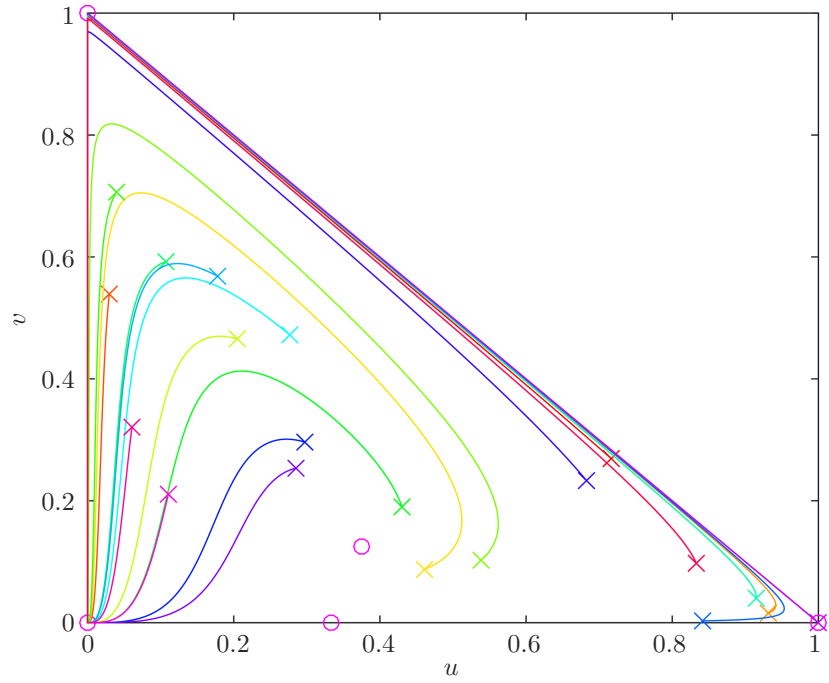}}
        \caption{$a=0.6$, $b=0.6$, $h=1$.}
    \end{subfigure}
    \\
    \begin{subfigure}{.9\linewidth}
        \resizebox{\linewidth}{!}{\input{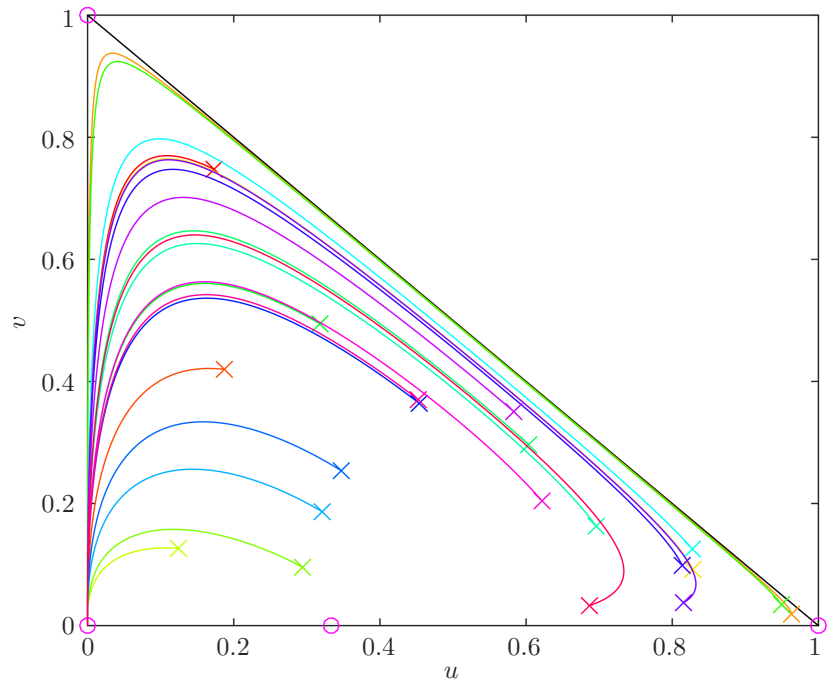}}
        \caption{$a=0.6$, $b=0.6$, $h=0.1$.}
    \end{subfigure}
    \caption{Trajectories of the diffusionless system with twenty random initial conditions and $b=a>\frac12$.
    When $h=1$, there is a coexistence state that is an unstable spiral. In both cases, $(0,0)$ seems to be
    globally attractive.}
    \label{fig:phase_portaits_large_b}
\end{figure}

In the limiting case $a=b$ with $h=1$, the following
properties can be established (quite directly). Of course they persist in a neighborhood of $a=b$, $h=1$.
\begin{enumerate}
    \item The dynamics in $\mathsf{T}$ are no longer of predator--prey type. The system remains 
    predative outside of the neighborhood of $(0,0)$ delimited by the graph of
    \[
    u\mapsto\frac12\left(2-u-\sqrt{\left(2-u\right)^2+4\left(u-\theta\right)}\right),
    \]
    but inside this neighborhood, the system is now cooperative. Even though $(0,0)$ might be globally 
    attractive inside the cooperative region (and this should actually be easy to establish), general solutions of
    the reaction--diffusion system might never enter uniformly this region. \textcolor{black}{In order to prove such a property},
    we need a Lyapunov function or a family of contractant sets, as in the proof of \thmref{convergence_to_0_if_b_is_small}. We did not manage 
    to perform such a construction.
    \item If \textcolor{black}{in addition} $a<\frac{2}{3}$, then there is a (unique) spatially uniform interior stationary state at
    \[
        \left(u^\star,v^\star\right)=\left(\frac{a}{4(1-a)},\frac{2-3a}{4(1-a)}\right).
    \]
    In view of \figref{phase_portaits_large_b}, this steady state is, regarding the diffusionless system, an 
    unstable spiral. This is one of the main obstacles encountered when trying to construct a Lyapunov
    function or a family a contractant sets.
\end{enumerate}

Because of these obstacles, we did not manage to find an analytical proof of the coextinction. Nevertheless,
in view of numerical experiments, it remains conjectured to hold true.

\subsection{On imperfect conversion efficiencies\label{subsec:discussion_efficiency}}
When the conversion efficiencies of the drive and the brake are no longer assumed to be perfect, the system
reads instead
\begin{equation*}
    \begin{cases}
\displaystyle        \mathscr{P}u = u\left( \frac{\left( 1-a \right)u+\left(1-c_B\right)\left(1-h_{BD}b\right)v+\left(2c_D \left( 1-a \right)+\left(1-c_D\right)\left(1-h_{DO}a\right)\right)\left( 1-u-v \right)}{w_{c_D,c_B}\left(u,v\right)}-1 \right),\medskip\\
\displaystyle        \mathscr{P}v = v\left( \frac{\left( 1-b \right)v+\left(2c_B \left( 1-b \right)+\left(1-c_B\right)\left(1-h_{BD}b\right)\right)u+\left( 1-hb \right)\left( 1-u-v \right)}{w_{c_D,c_B}\left(u,v\right)}-1 \right).
    \end{cases}
\end{equation*}
Here $c_D$, $c_B$, $h_{DO}$ and $h_{BD}$ are all in $\left[0,1\right]$ and are respectively the conversion
efficiency of the drive, the conversion efficiency of the brake, the dominance of $D$ on $O$ and the dominance
of $B$ on $D$. The normalizing mean fitness $w_{c_D,c_B}(u,v)$ is, similarly to the numerator, a mere
algebraic modification of the perfect case. 

It is therefore quite clear that the various uniform estimates on the reaction term we derived near 
the limit $b\sim 0$ should remain true for values of $(c_D,c_B)$ close enough to $(1,1)$ and
up to a slight modification of the threshold $a=\frac12$. 
Hence both \thmref{persistence_if_a_is_small} and \thmref{convergence_to_0_if_b_is_small}
are expected to remain true in this framework. 
We leave this extension for future work.

\subsection{On a slightly different model\label{subsec:discussion_Nagylaki}}
{\color{black}
When writing down the model, if we follow Nagylaki \cite{Nagylaki_1975} instead of following
Tanaka \textit{et al.} \cite{Tanaka_Stone_N}, we obtain a different system:
\begin{equation*}
    \begin{cases}
\displaystyle        \frac{\partial u}{\partial t} - \Delta u -2\nabla(\log  n)\cdot\nabla u = w_{OO}u\left( g_1\left( u,v \right)-w\left( u,v \right) \right),\medskip\\
\displaystyle        \frac{\partial v}{\partial t} - \Delta v -2\nabla(\log  n)\cdot\nabla v = w_{OO}v\left( g_2\left( u,v \right)-w\left( u,v \right) \right),\medskip\\
\displaystyle        \frac{\partial n}{\partial t} - \Delta n = \left( w_{OO}w(u,v)-1 \right)n,
    \end{cases}
\end{equation*}
where $w_{OO}$ is the fitness of wild-type homozygous individuals (it
can no longer be assumed to be unitary without loss of generality).
In short, the reaction term for $u$ and $v$ is 
multiplied by $w_{OO}w\left( u,v \right)$ and an explicit equation on $n$ appears.

It turns out that we can study the subsystem satisfied by 
$\left( u,v \right)$ exactly as before (handling $n$ as data) and obtain 
exactly the same results provided $w_{OO}$ is a positive constant.
The only (very small) differences appear in the algebraic expression of three thresholds, namely:
\begin{enumerate}
    \item the KPP threshold for the sole drive becomes $a=\frac{1}{3}$;
    \item the positivity threshold for the bistable speed of the sole drive becomes  
	$a=\frac{2}{3}$ (exactly);
    \item the cooperativity neighborhood when $a=b$, $h=1$ is now delimited by the straight line
	$u+v=\theta$.
\end{enumerate}
Note that, in this model, our coextinction result implies that, asymptotically in time,
the population density $n$ has an exponential growth with rate $w_{OO} -1$. 

The more general case where the wild-type fitness $w_{OO}\geq 0$ is non-constant (for instance, 
when it is a nonnegative function 
of $(u,v,n)$ that accounts for an Allee effect or a saturation effect) is left as an open problem. 
Nevertheless, let us point out that if the predator--prey
structure at $b=0$ is preserved (\textit{e.g.}, $w_{OO}$ depends only on
$n$), then the coextinction result still holds true (namely, there are only wild-type individuals 
in the long run) and this might simplify a lot the asymptotic population dynamics.

\subsection{On biological implications\label{subsec:discussion_bio}}
Whether the fitness cost associated to the drive allele, $a$, is greater or lower than $\frac{1}{2}$, is critical for the outcome of the model. The $a=\frac{1}{2}$ value is also critical in a well-mixed population in the absence of brake: when $a<\frac{1}{2}$, only the drive-only equilibrium ($u=1$) is locally stable, but when $a>\frac{1}{2}$, the wild-type-only equilibrium ($u=0$) becomes locally stable too: which equilibrium is eventually reached depends on the initial frequency of drive in the population (bistability). This frequency has to be higher than a threshold for the drive to fix. This bistability regime corresponds to what Tanaka \textit{et al.} \cite{Tanaka_Stone_N} call ``socially responsible drives''.  \cite{Tanaka_Stone_N} indeed showed that such drives could be stopped by barriers of finite width. Our results confirm that drives in the bistability parameter range (in our case, with $a>\frac{1}{2}$) are more ``socially responsible'' than drives without threshold ($a<\frac{1}{2}$): in the former case, a brake can stop the spread of a drive, while in the latter case, an indefinite co-invasion takes place, according to our model. 

Should artificial gene drives be implemented in nature for the control of wild populations, these results indicate that it would be preferable to design them so that they fall in the bistability parameter range. Such drives would be more controllable than drives without introduction thresholds. Having such control on parameters will obviously be much more complicated empirically, than it is theoretically. In addition, in the bistability parameter range, the threshold for successful drive introduction can be very high, meaning that very high numbers of drive-carrying individuals would need to be reared. 

Our macroscopic model ignores the effect of stochasticity by averaging all quantities. 
The spatial spread of the drive and the brake in individual-based models modeling smaller population sizes and possibly accounting for stochasticity remains to be investigated.

As our model uses a diffusion approximation, the location of the point of introduction of the brake does not affect our results. With non-diffusive models, for instance individual-based models, but also and more importantly in reality, this factor will probably affect the efficiency of the control of a gene drive by a brake.

\section*{Acknowledgments}
The authors thank three anonymous referees for valuable comments which lead to an improvement of the manuscript. 

This project has received funding from the European Research Council (ERC) under the European Union’s Horizon 2020 research and innovation program (grant agreement No 639638). 
This work was supported by a public grant as part of the Investissement d'avenir project, reference ANR-11-LABX-0056-LMH, LabEx LMH, and
ANR-14-ACHN-0003-01.

\appendix
\section{Weinberger's maximum principle \label{app:WMP}}
Below is recalled the main tool of the proof of \thmref{convergence_to_0_if_b_is_small}. For clarity, we temporarily get
rid of all our notations and adopt the original ones from \cite{Weinberger_1975}.

\begin{thm}[Weak maximum principle]
Let $D$ be a $C^{1,\nu}$ domain in $\mathbb{R}^n$ with $\nu\in\left(0,1\right)$, $S$ be a closed convex subset of $\mathbb{R}^m$,  $\mathbf{f}\left(\mathbf{u},x,t\right)$ be Lipschitz-continuous in $\mathbf{u}\in S$ and uniformly H\"{o}lder-continuous in 
$x\in D$ and $t\in\left[0,T\right]$, 
with the property that for any outward normal $\mathbf{p}$ at any boundary point $\mathbf{u}^\star$ of $S$, 
\[
\mathbf{p}\cdot\mathbf{f}\left(\mathbf{u}^\star,x,t\right)\leq 0\text{ for all }\left(x,t\right)\in D\times\left(0,T\right].
\]

Let 
\[
L=\frac{\partial}{\partial t}-\sum_{i,j=1}^n a^{ij}\left(x,t\right)\frac{\partial^2}{\partial x_i \partial x_j}
-\sum_{i=1}^n b_i\left(x,t\right)\frac{\partial}{\partial x_i}
\]
be uniformly parabolic with coefficients uniformly H\"{o}lder-continuous with H\"{o}lder exponent greater than $\frac12$.

If $\mathbf{u}$ is any solution in $D\times\left(0,T\right]$ of the system 
\[
Lu_{\alpha}=f_{\alpha}\left(\mathbf{u},x,t\right)\text{ for }\alpha=1,2,\dots,m
\]
which is continuous in $\overline{D}\times\left[0,T\right]$, and if the values of $\mathbf{u}$ 
on $\overline{D}\times\left\{0\right\}\cup \partial D\times\left[0,T\right]$ are bounded and H\"{o}lder-continuous and lie
in $S$, then $\mathbf{u}\left(x,t\right)\in S$ in $D\times\left(0,T\right]$.
\end{thm}

\begin{thm}[Strong maximum principle]
Let $D$ be an arbitrary domain in $\mathbb{R}^n$ and $S$ be a closed convex subset of $\mathbb{R}^m$ such that every boundary point of $S$ satisfies a
slab condition. Let $\mathbf{f}\left(\mathbf{u},x,t\right)$ be Lipschitz-continuous in $\mathbf{u}$, and
suppose that if $\mathbf{p}$ is any outward normal at a boundary point $\mathbf{u}^\star$, then 
\[
\mathbf{p}\cdot\mathbf{f}\left(\mathbf{u}^\star,x,t\right)\leq 0\text{ for all }\left(x,t\right)\in D\times\left(0,T\right].
\]

Let 
\[
L=\frac{\partial}{\partial t}-\sum_{i,j=1}^n a^{ij}\left(x,t\right)\frac{\partial^2}{\partial x_i \partial x_j}
-\sum_{i=1}^n b_i\left(x,t\right)\frac{\partial}{\partial x_i}
\]
be locally uniformly parabolic, and let its coefficients be locally bounded. 

If $\mathbf{u}$ is any solution in $D\times\left(0,T\right]$ of the system 
\[
Lu_{\alpha}=f_{\alpha}\left(\mathbf{u},x,t\right)\text{ for }\alpha=1,2,\dots,m
\]
with $\mathbf{u}\left(x,t\right)\in S$
and if $\mathbf{u}\left(x^\star,t^\star\right)\in\partial S$ for some $\left(x^\star,t^\star\right)\in D\times\left(0,T\right]$,
then $\mathbf{u}\left(x,t\right)\in\partial S$ in $D\times\left(0,t^\star\right]$.
\end{thm}

\section{Numerical scheme \label{app:numeric}}
The various numerical simulations presented earlier are all produced by the following Octave code
or slight variants of it.

{\small
\begin{Verbatim}
clear

% Biological parameters
a = 0.6; % fitness cost of the gene drive, u
b = 0.1; % fitness cost of the brake, v
h = 0.5; % dominance rate of B over O in the heterozygous OB

% Parameters of the numerical scheme
T = 300; % final time
L = 1280; % length of the spatial domain
M = 160000; % number of time steps
N = 16000; % number of spatial steps
dt = T/M; % size of the time step
dx = L/N; % size of the spatial step
X = [0:N]*dx; % discretized spatial domain
A = spdiags([ones(N+1,1) -2*ones(N+1,1) ones(N+1,1)],[-1, 0, 1],N+1,N+1); % 1D discrete Laplacian
A(1,1) = A(1,1)+1; % Neumann boundary condition on the left
A(end,end) = A(end,end)+1; % Neumann boundary condition on the right
B = eye(NX+1)-(dt/dx^2)*A; % matrix for the semi-implicit scheme

% Initial conditions for the allelic frequencies
U = [zeros(8*N/20+1,1);0.99*ones(3*N/20,1);zeros(9*N/20,1)]; % u at t=0
V = [zeros(9*N/20+1,1);0.001*ones(N/20,1);zeros(10*N/20,1)]; % v at t=0

% Draw the initial conditions as functions of space
plot(X,U,'-b',X,V,'-r')
axis([0 N*dx 0 1.1])
xlabel('$x$','fontsize',10)
ylabel('$(u,v)$','fontsize',10)

% Loop
for i=[1:M]
    w = 1-(a*U.^2+b*V.^2+2*b*U.*V)-2*(1-U-V).*(a*U+h*b*V);
    u = B\(U + dt*U.*(((1-a)*U+2*(1-a)*(1-U-V))./w-1));
    v = B\(V + dt*V.*(((1-b)*V+2*(1-b)*U+(1-h*b)*(1-U-V))./w-1));
    U = u;
    V = v;
    if mod(i,M/100)==0 % At T/100, 2*T/100, etc., draw the allelic frequencies
        plot(X,U,'-b',X,V,'-r')
        axis([0 N*dx 0 1.1])
        xlabel('$x$','fontsize',10)
        ylabel('$(u,v)$','fontsize',10)
        drawnow;
    endif
end
\end{Verbatim}
}}

\bibliographystyle{plain}
\bibliography{ref}

\end{document}